\documentclass[twoside,11pt]{article}
\usepackage{a4wide,cite,latexsym,amsfonts,amssymb,exscale,epsfig,hyperref,url}
\usepackage[centertags,sumlimits,intlimits,namelimits,reqno]{amsmath}

\usepackage{amsthm}
\theoremstyle{definition}
\newtheorem{theorem}{Theorem}[section]
\newtheorem{lemma}[theorem]{Lemma}
\newtheorem{corollary}[theorem]{Corollary}
\newtheorem{proposition}[theorem]{Proposition}
\newtheorem{definition}[theorem]{Definition}
\newtheorem{remark}[theorem]{Remark}

\usepackage{bbm}

\def\N{{\mathbbm N}}

\def\Z{{\mathbbm Z}}

\def\ssl{{\mathfrak{sl}}}

\def\ie{{\sl i.e.\/}}

\def\cf{{\sl c.f.\/}}

\let\phi=\varphi
\let\theta=\vartheta
\let\epsilon=\varepsilon

\def\tr{\mathop{\rm tr}\nolimits}

\def\ev{\mathop{\rm ev}\nolimits}
\def\coev{\mathop{\rm coev}\nolimits}

\def\dim{\mathop{\rm dim}\nolimits}

\def\id{\mathop{\rm id}\nolimits}
\def\Span{\mathop{\rm span}\nolimits}
\def\End{\mathop{\rm End}\nolimits}

\def\Hom{\mathop{\rm Hom}\nolimits}

\def\lim{\mathop{\rm lim}\limits}

\let\hat=\widehat
\let\tilde=\widetilde

\def\pprime{{\prime\prime}}
\def\ppprime{{\prime\prime\prime}}

\numberwithin{equation}{section}

\makeatletter

\newfont{\@aidxte}{cmsy10}
\newfont{\@aidxel}{cmsy10 scaled 1095}
\newfont{\@aidxtw}{cmsy10 scaled 1200}
\newlength\@aidxtexvi
\newlength\@aidxtexvii
\newlength\@aidxelxvi
\newlength\@aidxelxvii
\newlength\@aidxtwxvi
\newlength\@aidxtwxvii
\newcommand{\alignidx}[1]{%
\@aidxtexvi=\fontdimen16\@aidxte
\@aidxtexvii=\fontdimen17\@aidxte
\@aidxelxvi=\fontdimen16\@aidxel
\@aidxelxvii=\fontdimen17\@aidxel
\@aidxtwxvi=\fontdimen16\@aidxtw
\@aidxtwxvii=\fontdimen17\@aidxtw
{\mbox{$%
\fontdimen16\@aidxte=2.9pt
\fontdimen17\@aidxte=2.9pt
\fontdimen16\@aidxel=3.1pt
\fontdimen17\@aidxel=3.1pt
\fontdimen16\@aidxtw=3.3pt
\fontdimen17\@aidxtw=3.3pt
#1$}}%
\fontdimen16\@aidxte=\@aidxtexvi
\fontdimen17\@aidxte=\@aidxtexvii
\fontdimen16\@aidxel=\@aidxelxvi
\fontdimen17\@aidxel=\@aidxelxvii
\fontdimen16\@aidxtw=\@aidxtwxvi
\fontdimen17\@aidxtw=\@aidxtwxvii}

\makeatother

\newenvironment{myenumerate}{%
\begin{enumerate}
\setlength{\partopsep}{0pt}
\setlength{\parskip}{0pt}}{\end{enumerate}}

\def\nn{\notag}

\def\emph#1{{\sl #1\/}}
\def\sym#1{{\mathcal #1}}
\def\one{\mathbbm{1}}%
\def\bar#1{\overline{#1}}%

\def\op{\mathrm{op}}

\def\Vect{\mathbf{Vect}}
\def\fdVect{\mathbf{vect}}

\def\coend{\mathbf{coend}}

\def\msc#1{\noindent{\small Mathematics Subject Classification (2000):
#1\par}}%
\def\keywords#1{\noindent {\small keywords: #1\par}}%

\usepackage{tikz}
\usetikzlibrary{arrows}

\newenvironment{tikzformula}{%
\begin{aligned}\begin{tikzpicture}[scale=0.2]}{%
\end{tikzpicture}\end{aligned}}

\newcommand{\tikzsymb}[1]{%
\begin{tikzformula}#1\end{tikzformula}}

\newcommand{\leftcoev}[1]{%
\draw (3,0) arc (0:180:3);
\draw (5,0) arc (0:180:5);
\draw (3,0) -- (5,0);
\draw (-5,0) -- (-3,0);
\draw [very thick,->] (60:4) arc (60:120:4);
\draw (60:6.5) node {$#1$};
}

\newcommand{\rightcoev}[1]{%
\draw (3,0) arc (0:180:3);
\draw (5,0) arc (0:180:5);
\draw (3,0) -- (5,0);
\draw (-5,0) -- (-3,0);
\draw [very thick,->] (120:4) arc (120:60:4);
\draw (60:6.5) node {$#1$};
}

\newcommand{\leftev}[1]{%
\draw (3,0) arc (0:-180:3);
\draw (5,0) arc (0:-180:5);
\draw (3,0) -- (5,0);
\draw (-5,0) -- (-3,0);
\draw [very thick,->] (-60:4) arc (-60:-120:4);
\draw (-60:6.5) node {$#1$};
}

\newcommand{\rightev}[1]{%
\draw (3,0) arc (0:-180:3);
\draw (5,0) arc (0:-180:5);
\draw (3,0) -- (5,0);
\draw (-5,0) -- (-3,0);
\draw [very thick,->] (-120:4) arc (-120:-60:4);
\draw (-60:6.5) node {$#1$};
}

\newcommand{\identity}[1]{%
\draw (1,-4) -- (1,4) -- (-1,4) -- (-1,-4) -- cycle;
\draw [very thick,->] (0,2) -- (0,-2);
\draw (1.5,3) node [right] {$#1$};
}

\newcommand{\iddual}[1]{%
\draw (1,-4) -- (1,4) -- (-1,4) -- (-1,-4) -- cycle;
\draw [very thick,->] (0,-2) -- (0,2);
\draw (1.5,3) node [right] {$#1$};
}

\newcommand{\twistpos}[1]{%
\draw (- 1,-4) .. controls (- 1,-1) and (.8,-1) .. (.8,0);
\draw (-.8, 0) .. controls (-.8, 1) and ( 1, 1) .. ( 1,4);
\filldraw [fill=white,draw=white] (1.3,-4) .. controls (1.3,-1.3) and (-.5,- .7) .. (-.5,0)
-- (-1.1,0) .. controls (-1.1,- .7) and (.7,-1.3) .. (.7,-4) -- cycle;
\filldraw [fill=white,draw=white] (1.1, 0) .. controls (1.1,  .7) and (-.7, 1.3) .. (-.7,4)
-- (-1.3,4) .. controls (-1.3, 1.3) and (.5,  .7) .. (.5, 0) -- cycle;
\draw ( 1,-4) .. controls ( 1,-1) and (-.8,-1) .. (-.8,0);
\draw (.8, 0) .. controls (.8, 1) and (- 1, 1) .. (- 1,4);
\draw (-1, 4) -- (1, 4);
\draw (-1,-4) -- (1,-4);
\draw [very thick,->] (0,3.5) -- (0,2);
\draw (1.5,3) node [right] {$#1$};
}

\newcommand{\twistneg}[1]{%
\draw ( 1,-4) .. controls ( 1,-1) and (-.8,-1) .. (-.8,0);
\draw (.8, 0) .. controls (.8, 1) and (- 1, 1) .. (- 1,4);
\filldraw [fill=white,draw=white] (-1.3,-4) .. controls (-1.3,-1.3) and (.5,- .7) .. (.5,0)
-- (1.1,0) .. controls (1.1,- .7) and (-.7,-1.3) .. (-.7,-4) -- cycle;
\filldraw [fill=white,draw=white] (-1.1, 0) .. controls (-1.1,  .7) and (.7, 1.3) .. (.7,4)
-- (1.3,4) .. controls (1.3, 1.3) and (-.5,  .7) .. (-.5, 0) -- cycle;
\draw (- 1,-4) .. controls (- 1,-1) and (.8,-1) .. (.8,0);
\draw (-.8, 0) .. controls (-.8, 1) and ( 1, 1) .. ( 1,4);
\draw (-1, 4) -- (1, 4);
\draw (-1,-4) -- (1,-4);
\draw [very thick,->] (0,3.5) -- (0,2);
\draw (1.5,3) node [right] {$#1$};
}

\newcommand{\twistposdual}[1]{%
\draw (- 1,-4) .. controls (- 1,-1) and (.8,-1) .. (.8,0);
\draw (-.8, 0) .. controls (-.8, 1) and ( 1, 1) .. ( 1,4);
\filldraw [fill=white,draw=white] (1.3,-4) .. controls (1.3,-1.3) and (-.5,- .7) .. (-.5,0)
-- (-1.1,0) .. controls (-1.1,- .7) and (.7,-1.3) .. (.7,-4) -- cycle;
\filldraw [fill=white,draw=white] (1.1, 0) .. controls (1.1,  .7) and (-.7, 1.3) .. (-.7,4)
-- (-1.3,4) .. controls (-1.3, 1.3) and (.5,  .7) .. (.5, 0) -- cycle;
\draw ( 1,-4) .. controls ( 1,-1) and (-.8,-1) .. (-.8,0);
\draw (.8, 0) .. controls (.8, 1) and (- 1, 1) .. (- 1,4);
\draw (-1, 4) -- (1, 4);
\draw (-1,-4) -- (1,-4);
\draw [very thick,->] (0,2) -- (0,3.5);
\draw (1.5,3) node [right] {$#1$};
}

\newcommand{\twistnegdual}[1]{%
\draw ( 1,-4) .. controls ( 1,-1) and (-.8,-1) .. (-.8,0);
\draw (.8, 0) .. controls (.8, 1) and (- 1, 1) .. (- 1,4);
\filldraw [fill=white,draw=white] (-1.3,-4) .. controls (-1.3,-1.3) and (.5,- .7) .. (.5,0)
-- (1.1,0) .. controls (1.1,- .7) and (-.7,-1.3) .. (-.7,-4) -- cycle;
\filldraw [fill=white,draw=white] (-1.1, 0) .. controls (-1.1,  .7) and (.7, 1.3) .. (.7,4)
-- (1.3,4) .. controls (1.3, 1.3) and (-.5,  .7) .. (-.5, 0) -- cycle;
\draw (- 1,-4) .. controls (- 1,-1) and (.8,-1) .. (.8,0);
\draw (-.8, 0) .. controls (-.8, 1) and ( 1, 1) .. ( 1,4);
\draw (-1, 4) -- (1, 4);
\draw (-1,-4) -- (1,-4);
\draw [very thick,->] (0,2) -- (0,3.5);
\draw (1.5,3) node [right] {$#1$};
}

\newcommand{\braiddd}[2]{%
\draw (3,4) .. controls (3,2) and (-5,-1) .. (-5,-4)
-- (-3,-4) .. controls (-3,-2) and (5,1) .. (5,4) -- cycle;
\draw [very thick,->] (2,1.45) -- (-2,-1.45);
\filldraw [fill=white,draw=black] (-5,4) .. controls (-5,1) and (3,-2) .. (3,-4)
-- (5,-4) .. controls (5,-1) and (-3,2) .. (-3,4) -- cycle;
\draw [very thick,->] (-2,1.45) -- (2,-1.45);
\draw (-5.5,3) node [left] {$#1$};
\draw (5.5,3) node [right] {$#2$};
}

\newcommand{\braidinvdd}[2]{%
\draw (-5,4) .. controls (-5,1) and (3,-2) .. (3,-4)
-- (5,-4) .. controls (5,-1) and (-3,2) .. (-3,4) -- cycle;
\draw [very thick,->] (-2,1.45) -- (2,-1.45);
\filldraw [fill=white,draw=black] (3,4) .. controls (3,2) and (-5,-1) .. (-5,-4)
-- (-3,-4) .. controls (-3,-2) and (5,1) .. (5,4) -- cycle;
\draw [very thick,->] (2,1.45) -- (-2,-1.45);
\draw (-5.5,3) node [left] {$#1$};
\draw (5.5,3) node [right] {$#2$};
}

\newcommand{\righttwist}[1]{%
\draw (0,2) -- (0,0)
arc (0:-280:.5);
\draw (0,-2) .. controls (0,-1) and (0.5,-1) .. (0.5,0)
arc (0:50:.5);
\draw (0,2) node [right] {$#1$};
}

\newcommand{\lefttwist}[1]{%
\draw (0,2) -- (0,0)
arc (180:360:.5)
arc (0:100:.5);
\draw (0,-2) .. controls (0,-1) and (-0.5,-1) .. (-0.5,0)
arc (180:130:.5);
\draw (0,2) node [right] {$#1$};
}

\newcommand{\kirbyplus}[2]{%
\draw (1,2) -- (1,0)
arc (0:-280:.5);
\draw (1,-2) .. controls (1,-1) and (1.5,-1) .. (1.5,0)
arc (0:50:.5);
\draw (1,2) arc (0:75:1.5);
\draw (1,-2) arc (0:-180:1.5) -- (-2,2) arc (180:105:1.5);
\draw (-0.5,5) -- (-0.5,-3.1);
\draw (-0.5,-3.9) -- (-0.5,-5);
\draw (-0.5,6) node {$#1$};
\draw (1,2) node [right] {$#2$};
}

\newcommand{\kirbyminus}[2]{%
\draw (1,2) -- (1,0)
arc (180:360:.5)
arc (0:100:.5);
\draw (1,-2) .. controls (1,-1) and (0.5,-1) .. (0.5,0)
arc (180:130:.5);
\draw (1,2) arc (0:75:1.5);
\draw (1,-2) arc (0:-180:1.5) -- (-2,2) arc (180:105:1.5);
\draw (-0.5,5) -- (-0.5,-3.1);
\draw (-0.5,-3.9) -- (-0.5,-5);
\draw (-0.5,6) node {$#1$};
\draw (1,2) node [right] {$#2$};
}

\newcommand{\hopflink}[2]{%
\draw (-30:2) + (-2,0) arc (-30:305:2);
\draw (125:2) + (1,0) arc (125:-210:2);
\draw (-4.5,1) node {$#1$};
\draw (3,1) node [right] {$#2$};
}

\newcommand{\triplelink}[3]{%
\draw (-30:2) + (-3,0) arc (-30:270:2);
\draw (90:0.5) + (-3,-2.5) arc (90:-210:0.5);
\draw (290:2) + (-3,0) arc (290:305:2);
\draw (125:2) arc (120:55:2);
\draw (30:2) arc (30:-70:2);
\draw (150:2) arc (150:270:2);
\draw (90:0.5) + (0,-2.5) arc (90:-210:0.5);
\draw (235:2) + (3,0) arc (235:270:2);
\draw (90:0.5) + (3,-2.5) arc (90:-210:0.5);
\draw (-70:2) + (3,0) arc (-70:210:2);
\draw (-5.5,1) node {$#1$};
\draw (0,3) node {$#2$};
\draw (5,1) node [right] {$#3$};
}

\newcommand{\trace}[2]{%
\draw (-2,2) -- (-2,-2) -- (2,-2) -- (2,2) -- cycle;
\draw [->] (0,-2) arc (180:360:2) -- (4,0);
\draw (4,0) -- (4,2) arc (0:180:2);
\draw (7,2) node [left] {$#1$};
\draw (0,0) node {$#2$};
}

\newcommand{\partialtrace}[3]{%
\draw (-2.5,2) -- (-2.5,-2) -- (2.5,-2) -- (2.5,2) -- cycle;
\draw [->] (-1,4) -- (-1,3);
\draw (-1,3) -- (-1,2);
\draw [->] (-1,-2) -- (-1,-3);
\draw (-1,-3) -- (-1,-4);
\draw [->] (1,-2) arc (180:360:1.5) -- (4,0);
\draw (4,0) -- (4,2) arc (0:180:1.5);
\draw (-4,3.5) node [right] {$#1$};
\draw (7,2) node [left] {$#2$};
\draw (0,0) node {$#3$};
}

\newcommand{\inductionstart}[3]{%
\draw (-2,6) -- (-2,2) -- (2,2) -- (2,6) -- cycle;
\draw (0,2) -- (0,-2);
\draw (-2,-2) -- (-2,-6) -- (2,-6) -- (2,-2) -- cycle;
\draw [->] (0,-6) arc (180:360:2) -- (4,0);
\draw (4,0) -- (4,6) arc (0:180:2);
\draw (7,2) node [left] {$#1$};
\draw (0,4) node {$#2$};
\draw (0,-4) node {$#3$};
}

\newcommand{\inductionstep}[5]{%
\draw [->] (-1,4) -- (-1,3);
\draw (-1,3) -- (-1,2);
\draw (-3.5,2) -- (3.5,2) -- (3.5,-2) -- (-3.5,-2) -- cycle;
\draw (-1,-2) -- (-1,-13);
\draw (1,-2) -- (1,-4);
\draw (-0.5,-4) -- (5.5,-4) -- (5.5,-7) -- (-0.5,-7) -- cycle;
\draw (1,-7) -- (1,-8);
\draw (-0.5,-8) -- (5.5,-8) -- (5.5,-11) -- (-0.5,-11) -- cycle;
\draw [->] (1,-11) arc (180:270:2) -- (5,-13) arc (270:360:2) -- (7,0);
\draw (7,0) -- (7,2) arc (0:90:2) -- (3,4) arc (90:180:2);
\draw (-6,3.5) node [right] {$#1$};
\draw (10,2) node [left] {$#2$};
\draw (2.5,-5.5) node {$#3$};
\draw (2.5,-9.5) node {$#4$};
\draw (0,0) node {$#5$};
}

\begin{document}

\title{Weak Hopf Algebras unify the Hennings--Kauffman--Radford\\
and the Reshetikhin--Turaev invariant}
\author{Hendryk Pfeiffer\thanks{E-mail: \texttt{pfeiffer@math.ubc.ca}}}
\date{\small{Department of Mathematics, The University of British Columbia,\\
1984 Mathematics Road, Vancouver, BC, V2T 1Z2, Canada}\\[1ex]
February 24, 2012}

\maketitle

\begin{abstract}

We present an invariant of connected and oriented closed $3$-manifolds based
on a coribbon Weak Hopf Algebra $H$ with a suitable left-integral. Our
invariant can be understood as the generalization to Weak Hopf Algebras of
the Hennings--Kauffman--Radford evaluation of an unoriented framed link
using a dual quantum-trace. This quantum trace satisfies conditions that
render the link evaluation invariant under Kirby moves. If $H$ is a suitable
finite-dimensional Hopf algebra (not weak), our invariant reduces to the
Kauffman--Radford invariant for the dual of $H$. If $H$ is the Weak Hopf
Algebra Tannaka--Kre\v\i n reconstructed from a modular category $\sym{C}$,
our invariant agrees with the Reshetikhin--Turaev invariant. In particular,
the proof of invariance of the Reshetikhin--Turaev invariant becomes as
simple as that of the Kauffman--Radford invariant. Modularity of $\sym{C}$
is only used once in order to show that the invariant is non-zero; apart
from this, a fusion category with ribbon structure would be sufficient. Our
generalization of the Kauffman--Radford invariant for a Weak Hopf Algebra
$H$ and the Reshetikhin--Turaev invariant for its category of
finite-dimensional comodules $\sym{C}\simeq\sym{M}^H$ always agree by
construction. There is no need to consider a quotient of the representation
category modulo 'negligible morphisms' at any stage, and our construction
contains the Reshetikhin--Turaev invariant for an arbitrary modular category
$\sym{C}$, whether its relationship with some quantum group is known or not.
\end{abstract}

\msc{%
57M27,
16T05,
18D10
}

\section{Introduction}

In some special cases, the following two quantum invariants of connected and
oriented closed (smooth) $3$-manifolds are related: the invariant of
Reshetikhin--Turaev~\cite{ReTu91} and the invariant of
Hennings~\cite{He96}. Since the latter has been studied mainly for
finite-dimensional unimodular ribbon Hopf algebras, we focus on the
reformulation of the Hennings invariant according to Kauffman and
Radford~\cite{KaRa95}. This reformulation exploits the fact that these special
Hopf algebras have a unique cointegral with suitable properties, and this
results in a substantial simplification of Hennings' original construction.

It has already been demonstrated by Hennings~\cite{He96} that the
Reshetikhin--Turaev invariant for the modular categories associated with
$U_q(\ssl_2)$ at suitable roots of unity $q$ appears as a special case
of his construction. In order to show this, one needs to understand the
quotient of the category of tilting modules over $U_q(\ssl_2)$ modulo
so-called \emph{negligible} morphisms. This quotient is finitely semisimple
and has the structure of a modular category.

This example of Hennings raises the question of whether the agreement of the
Hennings invariant with the Reshetikhin--Turaev invariant is a coincidence or
whether there is an explanation in conceptual terms. One answer to this
question was given by Lyubashenko~\cite{Ly95}. Starting from the fusion
category with ribbon structure $\sym{C}$ used in the Reshetikhin--Turaev
invariant, he uses Majid's~\cite{Ma93} universal coend
$F=\coend(\sym{C},1_\sym{C})$ over the identity functor
$1_\sym{C}\colon\sym{C}\to\sym{C}$ which forms a Hopf algebra object
$F\in|\sym{C}|$ (\emph{braided group} in Majid's terminology). If
$\sym{C}\simeq{}_H\sym{M}$ is the category of modules over a
finite-dimensional Hopf algebra $H$, Lyubashenko's invariant for
$F\in|\sym{C}|$ coincides with the Kauffman--Radford invariant for $H$. Let us
summarize the idea of Lyubashenko's unification of the invariants as follows:
He forces the Hopf algebra $H$ of the Kauffman--Radford invariant into the
language of the Reshetikhin invariant by transmuting~\cite[Section 4]{Ma93} it
into a Hopf algebra object $F\in|\sym{C}|$.

The purpose of the present article is to reverse this approach and to present
a second way of unifying the Reshetikhin--Turaev with the Kauffman--Radford
invariant, this time by forcing the fusion category with ribbon structure
$\sym{C}$ into the language of the Kauffman--Radford invariant. Thanks to the
recent generalization of Tannaka--Kre\v\i n reconstruction to fusion
categories~\cite{Pf09a,Pf11}, we know that each fusion category $\sym{C}$, in
particular each modular category, is equivalent to the category of
finite-dimensional comodules $\sym{C}\simeq\sym{M}^H$ over a Weak Hopf Algebra
(WHA) $H=\coend(\sym{C},\omega)$, the universal coend over the long canonical
functor $\omega\colon\sym{C}\to\Vect_k$. We are therefore able to recover
$\sym{C}$ from $H$ and can therefore express the Reshetikhin--Turaev invariant
entirely in terms of $H$, \ie\ in the language of the Kauffman--Radford
invariant, something that is not possible in Lyubashenko's approach. One of
the advantages of this point of view is that it renders the proof of
invariance of the Reshetikhin--Turaev invariant as easy as that of the
Kauffman--Radford invariant whereas Lyubashenko's approach renders the proof
of invariance of the Kauffman--Radford invariance as difficult as that of the
Reshetikhin--Turaev invariant.

Let us now sketch how one can find a (co)algebra with extra structure $H$ such
that the Reshetikhin--Turaev invariant for $\sym{C}$ agrees with the
Kauffman--Redford invariant for $H$. It appears that all modular categories
that yield interesting $3$-manifold invariants, \ie\ invariants that are
stronger than invariants of homotopy type, have objects of non-integer
Frobenius--Perron dimension\footnote{This is an observation about the existing
literature, and we are not aware of any counter-example.}. These categories
therfore do not form the categories of modules over any Hopf algebra, see, for
example~\cite[Theorem 8.33]{EtNi05}.  The naive conjecture that the Hennings
invariant for a Hopf algebra $H$ might agree with the Reshetikhin--Turaev
invariant for the category ${}_H\sym{M}$ of modules over $H$, is therefore not
even well phrased.

Any conjecture on a coincidence of the Hennings with the Reshetikhin--Turaev
invariant has a chance of being true only if it is not ${}_H\sym{M}$ itself,
but rather a quotient of ${}_H\sym{M}$ modulo suitable negligible morphisms,
that forms the modular category. This is how Hennings' original example
works. But since there exist modular categories for which no relationship with
one of the standard quantum groups is known~\cite{HoRo08}, any conceptual
approach to relating the Hennings with the Reshetikhin--Turaev invariant needs
to avoid taking a quotient of the representation category.

Although the topologically interesting modular categories are not the
categories of modules over any Hopf algebra, every modular category is the
category of (co)modules over a Weak Hopf Algebra (WHA)~\cite{Pf09a} --- no
quotient required: for each modular category $\sym{C}$, one can
Tannaka--Kre\v\i n reconstruct a WHA $H$ whose category of finite-dimensional
comodules $\sym{M}^H$ is equivalent as a $k$-linear additive ribbon category
to the original modular category $\sym{C}$.

In the present article, we define an invariant of $3$-manifolds for a suitable
class of WHAs (see Theorem~\ref{thm_invariant} below). Our invariant can be
understood as the generalization of the Kauffman--Radford invariant from Hopf
algebras to WHAs. In the special case in which the WHA $H$ is a Hopf algebra,
our invariant reduces to the Kauffman--Radford invariant (for the
dual\footnote{For technical reasons, we work with comodules rather than
modules, and so our WHA $H$ corresponds to the dual of the Hopf algebra
featuring in the Kauffman--Radford invariant, and it is the category of
finite-dimensional comodules $\sym{M}^H$ of $H$ that forms the modular
category.} of $H$). In the special case in which $H$ is the WHA
reconstructed from a modular category $\sym{C}$, our invariant agrees with the
Reshetikhin--Turaev invariant for $\sym{C}$.

In fact, by construction, our generalization of the Kauffman--Radford
invariant for a WHA $H$ always agrees with the Reshetikhin--Turaev invariant
for the ribbon category $\sym{M}^H$. Modularity of $\sym{M}^H$ is sufficient
for the invariant to be non-zero. As a by-product, we obtain a new proof of
invariance of the Reshetikhin--Turaev invariant for an arbitrary modular
category $\sym{C}\simeq\sym{M}^H$ using computations in $H$ rather than
computations in $\sym{C}$. This new proof is substantially shorter than the
original proof presented in~\cite[Section II.3]{Tu10}.

In the near future, there will be a companion article relating the
Turaev--Viro invariant~\cite{TuVi92} with the Kuperberg invariant~\cite{Ku91}
along the same lines. Apart from tidying up some twenty-year-old results,
these identities between quantum invariants can be expected to prove useful if
one tries to categorify these invariants. Whereas categorifying a modular
category to some $2$-category with extra structure is still far beyond reach,
categorifying the reconstructed WHAs appears to be much more promising. In
particular, the WHAs reconstructed from the modular categories associated with
$U_q(\ssl_2)$ are substantially easier to understand than
$U_q(\ssl_2)$ at roots of unity itself together with the relevant
quotient of its category of modules.

The present article is structured as follows. In Section~\ref{sect_prelim}, we
briefly summarize some background material on WHAs and on Tannaka--Kre\v\i n
reconstruction. Section~\ref{sect_reconstructed} contains results on integrals
and cointegrals in the WHA reconstructed from a fusion category with ribbon
structure. We then define various ways of evaluating unoriented framed links
in $S^3$ in Section~\ref{sect_ribbon}. Our new invariant is presented in
Section~\ref{sect_invariant} in which we also show that in encompasses both
the Reshetikhin--Turaev invariant and the Kauffman--Radford
invariant. Appendix~\ref{app_wha} contains more details on WHAs with extra
structure, and Appendix~\ref{app_reconstruction} on their Tannaka--Kre\v\i n
reconstruction from fusion categories with extra structure.

\section{Preliminaries}
\label{sect_prelim}

In this section, we fix our notation and sketch some background
material on Weak Hopf Algebras, their categories of finite-dimensional
comodules, and on the canonical Weak Hopf Algebra associated to each
modular category via generalized Tannaka--Kre\v\i n reconstruction.

We use the following notation. If $\sym{C}$ is a category, we write
$X\in|\sym{C}|$ for the objects $X$ of $\sym{C}$, $\Hom(X,Y)$ for the
collection of all morphisms $f\colon X\to Y$ and $\End(X)=\Hom(X,X)$. We
denote the identity morphism of $X$ by $\id_X\colon X\to X$ and the
composition of morphisms $f\colon X\to Y$ and $g\colon Y\to Z$ by $g\circ
f\colon X\to Z$. If two objects $X,Y\in|\sym{C}|$ are isomorphic, we write
$X\cong Y$. If two categories $\sym{C}$ and $\sym{D}$ are equivalent, we write
$\sym{C}\simeq\sym{D}$. The identity functor on $\sym{C}$ is denoted by
$1_{\sym{C}}$. The category of vector spaces over a field $k$ is denoted by
$\Vect_k$ and its full subcategory of finite-dimensional vector spaces by
$\fdVect_k$. Both are $k$-linear, abelian and symmetric monoidal, and
$\fdVect_k$ is autonomous. The $n$-fold tensor power of some object
$X\in|\sym{C}|$ of a monoidal category
$(\sym{C},\otimes,\one,\alpha,\lambda,\rho)$ is denoted by $X^{\otimes n}$,
$n\in\N_0$. We set $X^{\otimes 0}:=\one$. We use the notation $\N$ and $\N_0$
for the positive integers and the non-negative integers, respectively.

\subsection{Weak Hopf Algebras and their corepresentations}

For the basics of Weak Bialgebras (WBAs) and Weak Hopf Algebras (WHAs), we
refer to~\cite{BoNi99,BoSz00} and to Appendix~\ref{app_wha}. In a WHA $H$ over
some field $k$, we denote by $\mu\colon H\otimes H\to H$, $\eta\colon k\to H$,
$\Delta\colon H\to H\otimes H$, $\epsilon\colon H\to k$ and $S\colon H\to H$
the multiplication, unit, comultiplication, counit and antipode,
respectively. The source and target counital maps are denoted by $\epsilon_s$
and $\epsilon_t$, and the source and target base algebras by $H_s$ and $H_t$,
respectively. The opposite comultiplication is given by
$\Delta^\op=\tau_{H,H}\circ\Delta$ where $\tau_{H,H}(x\otimes y)=y\otimes x$
for all $x,y\in H$.

The category of right $H$-comodules that are finite-dimensional over $k$, is
denoted by $\sym{M}^H$. It is a $k$-linear abelian and left-autonomous
monoidal category equipped with a $k$-linear, faithful and exact forgetful
functor $U^H\colon\sym{M}^H\to\fdVect_k$. This functor is in general not
strong monoidal and thereby not a fibre functor in the technical sense, but it
is equipped with a separable Frobenius structure.

A copivotal form $w\colon H\to k$ for $H$ is a dual group-like linear form
such that $S^2(x)=w(x^\prime)x^\pprime\bar w(x^\ppprime)$ for all $x\in H$,
\ie\ one that implements the square of the antipode by dual conjugation. Here
$\bar w\colon H\to k$ denotes the convolution inverse of $w$. We call such a
linear form $w$ \emph{copivotal} because it is this structure that renders the
category $\sym{M}^H$ a pivotal category.

The universal $r$-form of a coquasi-triangular WHA and its weak convolution
inverse are denoted by $r\colon H\otimes H\to k$ and $\bar r\colon H\otimes
H\to k$, respectively. Similarly, $\nu\colon H\to k$ and $\bar\nu\colon H\to
k$ denote the universal ribbon form and its convolution inverse in a coribbon
WHA. Recall that each coribbon WHA is copivotal with
$w(x)=v(x^\prime)\nu(x^\pprime)$ for all $x\in H$, involving the second dual
Drinfel'd element $v\colon H\to k, x\mapsto r(S(x^\prime)\otimes x^\pprime)$
and the universal ribbon form $\nu$.

For the convenience of the reader, we have collected in Appendix~\ref{app_wha}
the basic definitions and more detailed references to the literature as well
as the basic facts about the relevant additional structure on $\sym{M}^H$. For
example, if $H$ is coribbon, then $\sym{M}^H$ is a ribbon category. If $H$ is
finite-dimensional, split cosemisimple and pure, then $\sym{M}^H$ is fusion,
and if $H$ is in addition coribbon and weakly cofactorizable, then $\sym{M}^H$
forms a modular category.

\subsection{Tannaka--Kre\v\i n reconstruction}

For every multi-fusion category $\sym{C}$ that is $k$-linear over some
field $k$, there is a canonical functor
\begin{equation}
\label{eq_longfunctor}
\omega\colon\sym{C}\to\Vect_k,\quad X\mapsto \Hom_k(\hat V,\hat V\otimes X).
\end{equation}
Here we use the small progenerator
\begin{equation}
\label{eq_progenerator}
\hat V=\bigoplus_{j\in I}V_j,
\end{equation}
where the biproduct is over a set $I$ of one representative $V_j$, $j\in I$,
for each isomorphism class of simple objects of $\sym{C}$. The functor
$\omega$ is known as the \emph{long canonical functor}~\cite{Ha99b,Sz05} and
can be used in order to Tannaka--Kre\v\i n reconstruct a finite-dimensional
split cosemisimple coassociative counital coalgebra $H$ over $k$. It is given
by the universal coend $H=\coend(\sym{C},\omega)$ of $\omega$, and
$\sym{M}^H\simeq\sym{C}$ are equivalent as $k$-linear additive
categories. Since $\omega$ has a separable Frobenius
structure~\cite{Pf09a,Pf11}, $H$ forms a WHA, and $\sym{M}^H\simeq\sym{C}$ are
equivalent as monoidal categories as well.

If $\sym{C}$ carries a pivotal, spherical or ribbon structure, then $H$ is
copivotal, cospherical or coribbon, respectively. If $\sym{C}$ is fusion, then
$H$ is copure, and if $\sym{C}$ is modular, then $H$ is weakly
cofactorizable~\cite{Pf09a,Pf09b}. In all of these cases, the equivalence
$\sym{M}^H\simeq\sym{C}$ is compatible with the extra structure. In
Appendix~\ref{app_reconstruction}, we have compiled more details on how the
extra structure of $H$ is related to that of $\sym{C}$ and on how to perform
computations in $H$.

Note in particular that $\omega X=\Hom_k(\hat V,\hat V\otimes X)$ and
$\Hom_k(\hat V\otimes X,\hat V)$ are dually paired for all $X\in|\sym{C}|$,
see~\eqref{eq_gx}, and that most computations can be conveniently phrased in
terms of pairs of dual bases ${(e_m^{(X)})}_m$ and ${(e^m_{(X)})}_m$ of
$\omega X$ and ${(\omega X)}^\ast$, respectively. The vector space underlying
the reconstructed WHA $H$ is given in~\eqref{eq_coendvect}. Also note the
reconstruction of the copivotal form $w$ of~\eqref{eq_copivotal} and the
isomorphism $D_{\hat V}\in\End(\hat V)$ of~\eqref{eq_dtransformation}
involved.

\section{The reconstructed Weak Hopf Algebra}
\label{sect_reconstructed}

In order to see how our invariant encompasses the Reshetikhin--Turaev
invariant~\cite{ReTu91}, we need to develop some of the integral
theory of the WHA reconstructed from a modular category
$\sym{C}$. This is done in the present section. For background
material on the integral theory of WHAs, we refer to~\cite{BoNi99}.

Let $H$ be a coribbon WHA. A \emph{dual trace} is an element $\chi\in H$ such
that $\Delta^\op(\chi)=\Delta(\chi)$. It is called $S$-invariant if
$S(\chi)=\chi$. A \emph{dual quantum trace} is an element $t\in H$ such that
$\Delta^\op(t)=(\id_H\otimes S^2)\circ\Delta(t)$. A dual quantum trace $t\in
H$ is called $S$-\emph{compatible} if
\begin{equation}
\label{eq_scompatible}
\bar w(t^\prime)S(t^\pprime)=\bar w(t^\prime)t^\pprime.
\end{equation}
Observe that $t=w(\chi^\prime)\chi^\pprime$ is a dual quantum trace if and
only if $\chi$ is a dual trace. In this situation, $t$ is $S$-compatible if
and only if $\chi$ is $S$-invariant.

Note that for each $V\in|\sym{M}^H|$, its dual character $\chi_V\in H$
(see~\eqref{eq_dualchar}) forms a dual trace, and its dual quantum character
$T_V\in H$ (see~\eqref{eq_dualqchar}) is a dual quantum trace. Observe that
$S(\chi_V)=\chi_{V^\ast}$ and if $V^\ast\cong V$ in $\sym{M}^H$, then we have
in addition that $\chi_{V^\ast}=\chi_V$, \ie\ the dual character $\chi_V$ is
$S$-invariant.

An element $t\in H$ of a finite-dimensional WHA is called
\emph{non-degenerate} if $H^\ast\to k,\phi\mapsto\phi(t)$ is non-degenerate as
a functional on the dual WHA $H^\ast$, \ie\ its kernel does not contain any
non-zero left-ideal of $H^\ast$. This holds if and only if the bilinear form
$H^\ast\otimes H^\ast\to
k,\phi\otimes\psi\mapsto\phi(t^\prime)\psi(t^\pprime)$ is non-degenerate. A
\emph{left-integral} $\ell\in H$ is an element that satisfies
$x\ell=\epsilon_t(x)\ell$ for all $x\in H$. A \emph{right-integral} $r\in H$
is an element that satisfies $rx=r\epsilon_s(x)$ for all $x\in H$. A
\emph{two-sided integral} is both a left- and a right-integral. A
[left-,right-]cointegral of $H$ is a [left-,right-]integral of $H^\ast$.

A linear form $\zeta\colon H\to k$ is called \emph{dual central} if
$\zeta(x^\prime)x^\pprime=x^\prime\zeta(x^\pprime)$ for all $x\in H$. If $H$
is the WHA reconstructed from a multi-fusion category $\sym{C}$, then it is
split cosemisimple, and so its dual central linear forms can be computed as
follows.

\begin{proposition}
\label{prop_dualcentral}
Let $\sym{C}$ be a multi-fusion category over the field $k$ and
$H=\coend(\sym{C},\omega)$ be the finite-dimensional and split cosemisimple
WHA reconstructed from $\sym{C}$ using the long canonical
functor~\eqref{eq_longfunctor}. Then a basis for the vector space of all dual
central linear forms is given by ${\{\zeta_j\}}_{j\in I}$ where
\begin{equation}
\zeta_j({[\theta|v]}_{X}) = \left\{\begin{array}{ll}
\epsilon({[\theta|v]}_X),\quad&\mbox{if}\quad X\cong V_j,\\
0,\quad&\mbox{else},
\end{array}\right.
\end{equation}
for all simple $X\in|\sym{C}|$. A dual central linear form $\zeta=\sum_{j\in
I}c_j\zeta_j$ with coefficients $c_j\in k$ is convolution invertible if and
only if $c_j\neq 0$ for all $j\in I$.
\end{proposition}

The following theorem demonstrates that the canonical WHA reconstructed from a
spherical multi-fusion category $\sym{C}$ contains a very special
left-integral $\ell\in H$. This integral will feature in the construction of
the invariant below.

\begin{theorem}
\label{thm_integral}
Let $\sym{C}$ be a spherical multi-fusion category over $k$ and
$H=\coend(\sym{C},\omega)$ be the finite-dimensional and split cosemisimple
cospherical WHA reconstructed from $\sym{C}$ using the long canonical
functor~\eqref{eq_longfunctor}.
\begin{myenumerate}
\item
The following element $\ell\in H$ is a left-integral:
\begin{equation}
\label{eq_leftintegral}
\ell=\sum_{j\in I}\dim V_j\,\sum_m{[D^{-1}_{\hat V}\circ e^m_{(V_j)}\circ(D_{\hat V}\otimes\id_{V_j})|e^{(V_j)}_m]}_{V_j}.
\end{equation}
\item
The following linear form $c\colon H\to k$ is a two-sided cointegral:
\begin{equation}
\label{eq_cointegral}
c({[\theta|v]}_X) = \left\{
\begin{array}{ll}
\epsilon({[\theta|\rho^{-1}_{\hat V}]}_\one)\epsilon({[\rho_{\hat V}|v]}_\one) &
\quad\mbox{if}\quad X\cong\one,\\
0 &\quad\mbox{else},
\end{array}
\right.
\end{equation}
for all simple $X\in|\sym{C}|$.
\item
Both $\ell$ and $c$ are non-degenerate.
\item
The cointegral $c$ is $S$-invariant, and $\ell$ is a dual quantum
trace.
\item
The integral can be expressed as $\ell=\zeta(t_{\hat
V}^\prime)t_{\hat V}^\pprime=w(\chi_{\hat
V}^\prime)\zeta(\chi_{\hat V}^\pprime)\chi_{\hat V}^\ppprime$
where $\chi_{\hat V}\in H$ is the dual character~\eqref{eq_dualchar}
associated with the small progenerator $\hat V$ of~\eqref{eq_progenerator};
$t_{\hat V}=w(\chi_{\hat V}^\prime)\chi_{\hat V}^\pprime\in H$ is its dual
quantum character~\eqref{eq_dualqchar}; and $\zeta\colon H\to k$ is the
convolution invertible and dual central linear form given by
\begin{equation}
\label{eq_dualcentral}
\zeta({[\theta|v]}_X) = (\dim X)\,\epsilon({[\theta|v]}_X)
\end{equation}
for all simple $X\in|\sym{C}|$.
\item
The integral $\ell$ is $S$-compatible.
\end{myenumerate}
\end{theorem}

\begin{proof}
\begin{myenumerate}
\item
Note that the canonical left-integral $\ell_{\mathrm{can}}\in H$
that exists in every WHA, turns out to be a multiple of our $\ell$:
\begin{equation}
\ell_{\mathrm{can}} = \sum_jb_j^\prime\beta_j(S^2(b_j^\pprime))=|I|\,\ell.
\end{equation}
Here we have written $\sum_jb_j\otimes\beta_j$ for the canonical
element in $H\otimes H^\ast$. Therefore, $\ell_{\mathrm{can}}=0$
whenever the characteristic of $k$ divides the number $|I|$ of
isomorphism classes of the simple objects of $\sym{C}$. Our integral
$\ell\in H$ avoids this problem and never vanishes as we show in Part~(3)
below. The proof that it indeed forms a left-integral is by a direct
calculation and is most transparent if one uses the isomorphism
$H\cong\End({\hat V}^\ast\otimes \hat V)$ of~\cite[Section~4.1]{Pf09b}.
\item
By the result dual to~\cite[Lemma 3.3]{BoNi99}, the set of
right-cointegrals of $H$ is isomorphic as a left-$H$-comodule to
$\Hom_{\sym{M}^H}(H,H_s)$ where both $H$ and $H_s\cong\one$ are viewed as
right-$H$-comodules. Since $H$ is split cosemisimple, the set
$\Hom_{\sym{M}^H}(H,H_s)$ and thereby the set of right-cointegrals is
known explicitly. A direct computation shows that such a right-cointegral is
two-sided if and only if it is a scalar multiple of $c$ of~\eqref{eq_cointegral}.
\item
A direct computation shows that
$c\rightharpoonup\ell=1$. By~\cite[Theorem~3.18]{BoNi99}, both
$\ell$ and $c$ are therefore non-degenerate. In particular, $\ell\neq
0$. The theorem also implies that $\ell\rightharpoonup c=1^\ast$, where
$1^\ast\in H^\ast$ is the unit of the dual WHA, a result that is needed in
Part~(4) below.
\item
Since there exists the two-sided non-degenerate cointegral $c$ of
$H$, by the result dual to~\cite[Lemma~3.21]{BoNi99}, all two-sided
cointegrals are $S$-invariant. In particular, $c$ is. Since $c$ is a
non-degenerate left-cointegral and $\ell\rightharpoonup c=1^\ast$,
$\ell$ is its dual left-integral, and we can apply the result dual
to~\cite[Theorem~3.20]{BoNi99}. Since $c$ is $S$-invariant, this
theorem implies that $\ell$ is a dual quantum trace.
\item
Since the coalgebra underlying $H$ is finite-dimensional and split
cosemisimple, we know the coefficients of the right-$H$ comodule $\hat
V\in|\sym{M}^H|$ and can compute its dual character~\eqref{eq_dualchar}:
\begin{equation}
\chi_{\hat V}
= \sum_{j\in I}\chi_{V_j}
= \sum_{j\in I}\sum_m{[e^m_{(V_j)}|e_m^{(V_j)}]}_{V_j}\in H.
\end{equation}
A direct computation using the copivotal form~\eqref{eq_copivotal} and
equation~\eqref{eq_dualcentral}, proves the claim.
\item
Since ${\hat V}^\ast\cong\hat V$, its dual character is
$S$-invariant, \ie\ $S(\chi_{\hat V})=\chi_{\hat V}$. In order to prove the
claim, we use that $\zeta\circ S=\zeta$.
\end{myenumerate}
\end{proof}

\section{Evaluation of (unoriented) framed links}
\label{sect_ribbon}

\subsection{Ribbon diagrams}
\label{sect_evaluation}

Let $\sym{C}$ be a ribbon category (see
Appendix~\ref{app_coribbon}). Every morphism of $\sym{C}$ can be
represented by a composition of tensor products of the following
string diagrams,
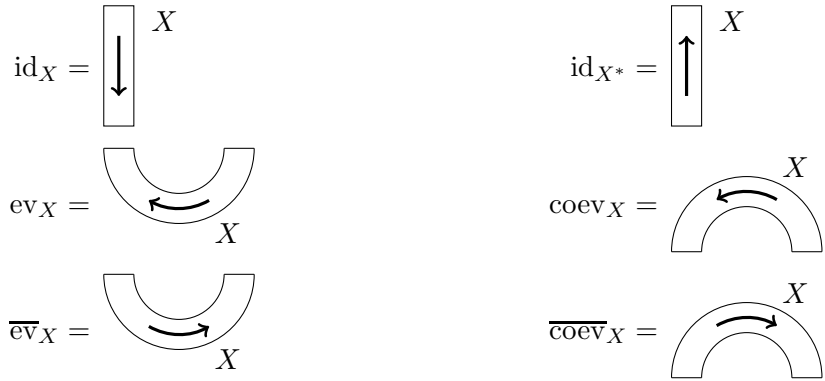
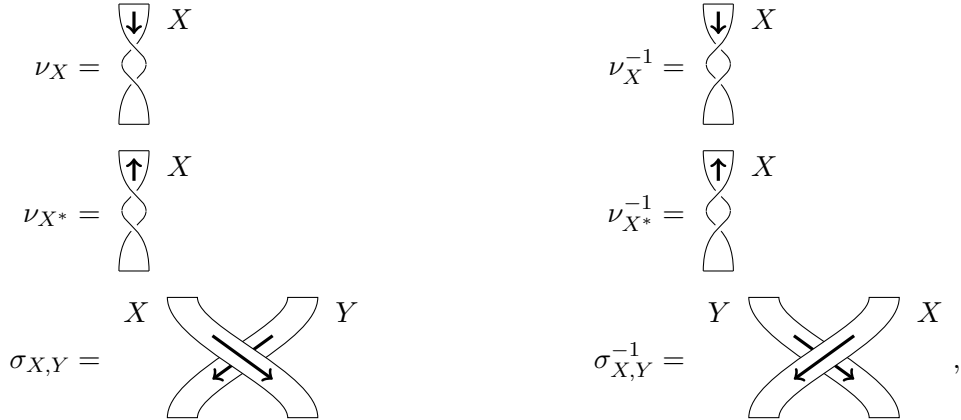
\begin{align}
\id_X             &= \tikzsymb{\identity{X}}\qquad\qquad&
\id_{X^\ast}      &= \tikzsymb{\iddual{X}}\nn\\
\ev_X             &= \tikzsymb{\leftev{X}}\qquad\qquad&
\coev_X           &= \tikzsymb{\leftcoev{X}}\nn\\
\bar\ev_X         &= \tikzsymb{\rightev{X}}\qquad\qquad&
\bar\coev_X       &= \tikzsymb{\rightcoev{X}}\nn
\end{align}
\begin{align}
\nu_X             &= \tikzsymb{\twistpos{X}}\qquad\qquad&
\nu^{-1}_X        &= \tikzsymb{\twistneg{X}}\nn\\
\nu_{X^\ast}      &= \tikzsymb{\twistposdual{X}}\qquad\qquad&
\nu^{-1}_{X^\ast} &= \tikzsymb{\twistnegdual{X}}\nn\\
\sigma_{X,Y}      &= \tikzsymb{\braiddd{X}{Y}}\qquad\qquad&
\sigma^{-1}_{X,Y} &= \tikzsymb{\braidinvdd{Y}{X}},\nn
\end{align}
in which the components of the ribbon are labeled by objects
$X,Y,\ldots\in|\sym{C}|$. The monoidal unit $\one\in|\sym{C}|$ is invisible in
these diagrams. If an object label $X\in|\sym{C}|$ is replaced by its dual
$X^\ast$, the arrow is reversed. Note that we read composition from top to
bottom and the tensor product from left to right. This agrees with about half
of the literature, but notably differs from Turaev~\cite{Tu10} who reads
composition from the bottom up and who calls the right-handed rather than the
left-handed twist $\nu_X$. Our choice of diagrams turns out to be convenient
in the present context as we study corepresentations rather than
representations of (Weak) Hopf Algebras.

By a result of Reshetikhin--Turaev~\cite{ReTu90} which can be viewed as a
coherence theorem for ribbon categories, every plane projection of an oriented
framed tangle in $S^3$ can be arranged to agree with such a diagram, but
without labels. If one now labels the components of the tangle with objects of
$\sym{C}$, the resulting morphism of $\sym{C}$ can be shown to be independent
of the chosen projection~\cite{ReTu90}. Therefore, a given ribbon category
$\sym{C}$ associates morphisms of $\sym{C}$ with labeled oriented framed
tangles and, more specially, endomorphisms of the monoidal unit
$\one\in|\sym{C}|$ with labeled oriented framed links.

\subsection{Reshetikhin--Turaev evaluation for ribbon categories}

Let $\sym{C}$ be a ribbon category, $V\in|\sym{C}|$ be an object of $\sym{C}$
such that $V^\ast\cong V$, and $\zeta\colon 1_{\sym{C}}\Rightarrow
1_{\sym{C}}$ be a natural isomorphism of the identity functor such that
$\zeta^\ast=\zeta$. Given a plane projection of an (unoriented) framed link
$L$ in $S^3$, we label all components by $V$ and insert $\zeta_V$ somewhere
(anywhere) into each component of the link. We call the resulting morphism
\begin{equation}
\label{eq_rtevaluation}
{\left<L\right>}^{(\sym{C})}_{V,\zeta}\colon\one\to\one,
\end{equation}
the \emph{Reshetikhin--Turaev evaluation} of the link $L$. Notice
that~\eqref{eq_rtevaluation} is independent of the orientation of each
component of $L$ and therefore well defined. The link
evaluation~\eqref{eq_rtevaluation} for $V=\hat V$ of~\eqref{eq_progenerator}
and $\zeta$ the natural transformation associated with the linear form $\zeta$
of~\eqref{eq_dualcentral} is the one that features in the Reshetikhin--Turaev
invariant~\cite{ReTu91}.

\subsection{Evaluation for coribbon Weak Hopf Algebras}

In this section, we evaluate an (unoriented) framed link in the ribbon
category $\sym{M}^H$ for a suitable coribbon WHA $H$.

\begin{remark}[Summary of results from~\cite{Pf09a}]
\label{rem_coribbon}
Let $H$ be a coribbon WHA over some field $k$. Then the category $\sym{M}^H$
of finite-dimensional right $H$-comodules is a ribbon category with
\begin{alignat}{3}
\ev_V            &\colon V^\ast\otimes V\to H_s,  &&\quad \theta\otimes v\to\theta(v_0)\epsilon_s(v_1),\\
\coev_V          &\colon H_s\to V\otimes V^\ast,  &&\quad x\mapsto ({(e_j)}_0\otimes e^j)\epsilon(x{(e_j)}_1),\\
\nu_V            &\colon V\to V,                  &&\quad v\mapsto v_0\nu(v_1),\\
\nu^{-1}_V       &\colon V\to V,                  &&\quad v\mapsto v_0\bar\nu(v_1),\\
\sigma_{V,W}     &\colon V\otimes W\to W\otimes V,&&\quad v\otimes w\mapsto (w_0\otimes v_0)r(w_1\otimes v_1),\\
\sigma^{-1}_{V,W}&\colon W\otimes V\to V\otimes W,&&\quad w\otimes v\mapsto (v_0\otimes w_0)\bar r(w_1\otimes v_1),\\
\end{alignat}
where we have used Sweedler notation for comodules (see
Appendix~\ref{app_wba}), $e_j\otimes e^j\in V\otimes V^\ast$ denotes the
canonical element, and $H_s$ plays the role of the monoidal unit of
$\sym{M}^H$.

The right evaluation and coevaluation $\bar\ev_V$ and $\bar\coev_V$
can be computed as in~\eqref{eq_barevribbon}
and~\eqref{eq_barcoevribbon}, respectively. Replacing one object
$V\in|\sym{M}^H|$ by its dual is is done as in~\eqref{eq_dualaction}.
\end{remark}

In the remainder of this section, we show how the Reshetikhin--Turaev
evaluation ${\left<L\right>}^{(\sym{M}^H)}_{V,\zeta}$ can be computed using
only the WHA $H$. In this computation, the dual quantum character $T_V$
appears once for each component of $L$, in conjunction with the linear form
associated with $\zeta$. The element
\begin{equation}
\ell=\zeta(T_V^\prime)T_V^\pprime\in H
\end{equation}
turns out to be an $S$-invariant dual quantum trace.

\begin{proposition}
\label{prop_universalnat}
Let $H$ be a WBA and $U^H\colon\sym{M}^H\to\Vect_k$ be the usual forgetful
functor. Every natural transformation $f\colon 1_{\sym{M}^H}\Rightarrow
1_{\sym{M}^H}$ is of the form
\begin{equation}
\label{eq_universalnat}
f_V(v) = v_0\alpha^{(f)}(v_1)
\end{equation}
for all $V\in|\sym{M}^H|$ and $v\in V$. Here $\alpha^{(f)}\colon H\to k$ is a
uniquely determined dual central linear form. In addition, $f\colon
1_{\sym{M}^H}\Rightarrow 1_{\sym{M}^H}$ is a natural equivalence if and only
if $\alpha^{(f)}$ is convolution invertible.
\end{proposition}

\begin{proof}
By the universal property of the universal coend
$H\cong\coend(\sym{M}^H,U^H)$, condition~\eqref{eq_universalnat} defines a unique linear
form $\alpha^{(f)}$ for the natural transformation $\alpha$. Since each
$f_V\colon V\to V$ is a morphism, $\alpha^{(f)}$ is dual central. The
convolution inverse $\overline{\alpha^{(f)}}$ is given by $f^{-1}_V(v) =
v_0\overline{\alpha^{(f)}}(v_1)$ if it exists, again using the universal
property of the coend.
\end{proof}

\begin{proposition}
Let $H$ be a WBA and $U^H\colon\sym{M}^H\to\Vect_k$ be the usual forgetful
functor. Every natural transformation $f\colon-\otimes-\Rightarrow-\otimes-$
of the functor
$-\otimes-\colon\sym{M}^H\times\sym{M}^H\to\sym{M}^H\times\sym{M}^H$ is of the
form
\begin{equation}
\label{eq_universalnat2}
f_{V,W}(v\otimes w) = v_0\otimes w_0\,\alpha^{(f)}(v_1\otimes w_1)
\end{equation}
for all $V,W\in|\sym{M}^H|$ and $v\in V$, $w\in W$. Here $\alpha^{(f)}\colon
H\otimes H\to k$ is a uniquely determined linear form that satisfies
\begin{gather}
\label{eq_universalnata}
\epsilon(x^\prime y^\prime)\alpha(x^\pprime\otimes y^\pprime)
= \alpha(x\otimes y)
= \alpha(x^\prime\otimes y^\prime)\epsilon(x^\pprime  y^\pprime),\\
\label{eq_universalnatb}
x^\prime y^\prime\,\alpha(x^\pprime\otimes y^\pprime)
= \alpha(x^\prime\otimes y^\prime)\,x^\pprime y^\pprime,
\end{gather}
for all $x,y\in H$.
\end{proposition}

\begin{proof}
By the universal property of the universal coend $H\otimes
H\cong\coend(\sym{M}^H\times\sym{M}^H,U^H\otimes U^H)$,
\eqref{eq_universalnat2} defines a unique linear form $\alpha^{(f)}$ for the
natural transformation $\alpha$, and~\eqref{eq_universalnata} is
satisfied. Since each $f_{V,W}\colon V\otimes W\to V\otimes W$ is a morphism,
condition~\eqref{eq_universalnatb} holds.
\end{proof}

In the following, we use the same symbol for the natural transformation and
for the associated linear form, for example, $\alpha\colon
1_{\sym{M}^H}\Rightarrow 1_{\sym{M}^H}$ and $\alpha\colon H\to k$ or
$\beta\colon-\otimes-\Rightarrow-\otimes-$ and $\beta\colon H\otimes H\to k$.

\begin{lemma}
\label{la_partialtrace}
Let $H$ be a coribbon WHA and $f\colon X\otimes V\to X\otimes V$ be a morphism
in $\sym{M}^H$. Then
\begin{equation}
\tikzsymb{\partialtrace{X}{V}{f}}(x)
= \sum_{i,j} g_i(x)\otimes e^j({(h_i(e_j)}_0)w({(h_i(e_j))}_1)
\end{equation}
for all $x\in X$ where we have written $f=\sum_i g_i\otimes h_i$ with
$g_i\in\End(X)$ and $h_i\in\End(V)$, $\coev_V=\sum_j e_j\otimes e^j$, and
where $w\colon H\to k$ denotes the copivotal form. The diagram is in the
ribbon category $\sym{M}^H$ and is drawn in blackboard framing.
\end{lemma}

\begin{proof}
Use the definitions in Remark~\ref{rem_coribbon} as well as
\begin{equation}
(x_0\otimes v_0)\epsilon(\epsilon_s(x_1)v_1) = x\otimes v
\end{equation}
and
\begin{equation}
(x_0\otimes \theta_0)\epsilon(x_1\epsilon_s(\theta_1)) = x\otimes\theta
\end{equation}
for all $x\in X$, $v\in V$, and $\theta\in V^\ast$.
\end{proof}

\begin{proposition}
\label{prop_tracelemma}
Let $H$ be a coribbon WHA and $\alpha\colon-\otimes-\Rightarrow-\otimes-$ be a
natural transformation. Then for all $X,V\in|\sym{M}^H|$,
\begin{equation}
\label{eq_tracelemma}
\tikzsymb{\partialtrace{X}{V}{\alpha_{X,V}}}(x)
= x_0\,\alpha(x_1\otimes T_V)
\end{equation}
for all $x\in X$. The diagram is in $\sym{M}^H$ and drawn in blackboard
framing, and $T_V\in H$ denotes the dual quantum character
of~\eqref{eq_dualqchar}.
\end{proposition}

\begin{proof}
Apply Lemma~\ref{la_partialtrace} to $f_{X,V}=\alpha_{X,V}$.
\end{proof}

\begin{lemma}
\label{la_tracelemma}
Let $H$ be a coribbon WHA and $\alpha\colon 1_{\sym{M}^H}\Rightarrow
1_{\sym{M}^H}$ be a natural transformation. Then for $V\in|\sym{M}^H|$,
\begin{equation}
\label{eq_trace}
\tikzsymb{\trace{V}{\alpha_V}}(h) = \alpha(T_V^\prime)\epsilon_s(h\epsilon_t(T_V^\pprime))
\end{equation}
for all $h\in H_s$.
\end{lemma}

\begin{proof}
Recall that the trace~\eqref{eq_trace} is a linear map $H_s\to H_s$ where
$H_s\cong\one\in|\sym{M}^H|$ is the monoidal unit. Using the definitions in
Remark~\ref{rem_coribbon}, we obtain
\begin{equation}
\eqref{eq_trace} = \alpha(T_V^\prime)\epsilon(hT_V^\pprime)\epsilon_s(S(T_V^\ppprime))
\end{equation}
which can be shown to agree with the right hand side of~\eqref{eq_trace}.
\end{proof}

\begin{theorem}
\label{thm_evaluationform}
Let $H$ be a coribbon WHA, $L$ be an (unoriented) framed link in $S^3$ with
$m$ components, $V\in|\sym{M}^H|$ and $\zeta\colon 1_{\sym{M}^H}\Rightarrow
1_{\sym{M}^H}$ be a natural equivalence such that $\zeta^\ast=\zeta$. Then the
Reshetikhin--Turaev evaluation is of the form
\begin{equation}
\label{eq_computeeval}
{\left<L\right>}^{(\sym{M}^H)}_{V,\zeta}(h) = \phi^{(L)}(h\otimes\underbrace{\ell\otimes\cdots\otimes\ell}_{m})
\end{equation}
for all $h\in H_s$ with a linear map $\phi^{(L)}\colon H_s\otimes H^{\otimes m}\to H_s$.
Here, the element
\begin{equation}
\ell = \zeta(T_V^\prime) T_V^\pprime\in H
\end{equation}
is an $S$-compatible dual quantum trace.
\end{theorem}

\begin{proof}
We prove a slightly stronger claim in which we insert into each component
$L_j$, $1\leq j\leq m$, of $L$ a natural transformation $\gamma^{(j)}_V\colon
V\to V$ that is of the form $\gamma_V^{(j)}=\xi^{(j)}_V\circ\zeta_V$ with
arbitrary $\xi_V^{(j)}\colon V\to V$. The proof proceeds by induction on the
number of components $m$.

If $m=1$, the link evaluation is of the form
\begin{equation}
{\left<L\right>}^{(\sym{M}^H)}_{V,\zeta} = \tikzsymb{\inductionstart{V}{\gamma_V^{(1)}}{\nu_V^z}}
\end{equation}
where $\nu_V^z$, $z\in\Z$, is the appropriate power of the twist. Putting
$\alpha_V=\nu_V^z\circ\gamma^{(1)}_V=\nu_V^z\circ\xi^{(1)}_V\circ\zeta_V$
in Lemma~\ref{la_tracelemma} shows that
\begin{equation}
{\left<L\right>}^{(\sym{M}^H)}_{V,\zeta}(h)
= \zeta(T_V^\prime)\psi^{(1)}(T_V^\pprime)\epsilon_s(h\epsilon_t(T_V^\ppprime))
= \psi^{(1)}(\ell^\prime)\epsilon_s(h\epsilon_t(\ell^\pprime))
\end{equation}
for all $h\in H_s$ where the linear form $\psi^{(1)}\colon H\to k$ implements
the natural transformation $\psi^{(1)}_V=\nu_V^z\circ\xi^{(1)}_V$. This proves
our stronger proposition. For $\xi^{(1)}_V=\id_V$, we obtain the claim of the
theorem for $m=1$ as a special case.

We now assume that our stronger assumption holds for some $m\in\N$ and consider
a link $L$ with $m+1$ components. If we select one of the components, labeled
$V$ and, without loss of generality, numbered $m+1$, the diagram of $L$ can be
arranged in such a way that the selected component $V$ appears in the
following fashion:
\begin{equation}
f_{V^{\otimes m}} = \tikzsymb{\inductionstep{V^{\otimes m}}{V}{\gamma^{(m+1)}_V}{\nu_V^z}{\beta_{V^{\otimes m},V}}},
\end{equation}
with some natural transformation $\beta\colon-\otimes-\Rightarrow-\otimes-$
and some $z\in\Z$. We now apply Proposition~\ref{prop_tracelemma} with
$\alpha_{V^{\otimes m},V}=(\id_{V^{\otimes
m}}\otimes(\nu_V^z\circ\gamma^{(m+1)}_V))\circ\beta_{V^{\otimes m},V}$ and
$X=V^{\otimes m}$. This proposition shows that
\begin{equation}
f_{V^{\otimes m}}(x) = x_0\,\omega(x_1),
\end{equation}
for all $x\in V^{\otimes m}$ where
\begin{equation}
\label{eq_omega}
\omega(y)=\zeta(T_V^\prime)\psi^{(m+1)}(T_V^\pprime)\alpha(y\otimes T_V^\ppprime)
=\psi^{(m+1)}(\ell^\prime)\alpha(y\otimes\ell^\pprime)
\end{equation}
for all $y\in H$. Here, $\psi^{(m+1)}\colon H\to k$ is the linear form that
implements the natural transformation
$\psi^{(m+1)}_V=\nu_V^z\circ\xi^{(m+1)}$. In particular, the map
$f_{V^{\otimes m}}$ is natural in $V^{\otimes m}$. Even stronger, we can split
$X=V^{\otimes m}=Y\otimes V$ with $Y=V^{\otimes(m-1)}$ and see that
$f_{Y\otimes V}(y\otimes v)=y_0\otimes v_0\,\omega(y_1v_1)$ for all $y\in Y$
and $v\in V$ which shows that $f_{Y\otimes V}$ is natural both in $Y$ and in
$V$. By induction, we see that $f_{V^{\otimes m}}$ is natural in each component
$V$ of the tensor power $V^{\otimes m}$.

The evaluation ${\left<L\right>}^{(\sym{M}^H)}_{V,\zeta}$ for the
$(m+1)$-component link is therefore equal to the evaluation for an
$m$-component link in which we have inserted different natural transformations
into every component: the composition of $f_{V^{\otimes m}}$ which is natural
in each tensor factor $V$ with the $\zeta_V$ from the original claim. By the
assumption of our induction, ${\left<L\right>}^{(\sym{M}^H)}_{V,\zeta}$ is of
the form of our stronger claim. Note that for each component of $L$,
\eqref{eq_omega} is applied once and yields one tensor factor $\ell$.

Depending on the order in which we consider the components of $L$, we may
obtain different linear maps $\phi^{(L)}$. The right hand side
of~\eqref{eq_computeeval}, however, always agrees with the Reshetikhin--Turaev
evaluation.
\end{proof}

\subsection{Generalized dual Hennings--Kauffman--Radford evaluation}

The idea of Hennings~\cite{He96} in the case of Hopf algebras (not weak) is
that the $S$-compatible dual quantum trace $\ell\in H$ that features in
Theorem~\ref{thm_evaluationform} can be replaced by an arbitrary
$S$-compatible dual quantum trace and still yields a well-defined link
evaluation, \ie\ an evaluation that gives the same value irrespective of the
diagram that is used in order to represent the link. The same holds for
WHAs. We summarize this result in the following.

\begin{definition}
Let $H$ be a coribbon WHA, $\ell\in H$ be an $S$-compatible dual quantum
trace, and $L$ be an $m$-component (unoriented) framed link in $S^3$. The
\emph{generalized Hennings--Kauffman--Radford evaluation}
${\left<L\right>}^{(H)}_\ell$ is defined as follows. Consider a diagram of
$L$. Label each component of $L$ with a formal symbol $\sym{X}$ which stands
for a finite-dimensional right $H$-comodule with coaction
$\beta_\sym{X}\colon\sym{X}\to\sym{X}\otimes H$. We impose only the relations
that $\sym{X}$ be a rigid object of $\Vect_k$ and that $\beta_\sym{X}$ be a
right-$H$ comodule, \ie\
\begin{eqnarray}
\alpha_{\sym{X},H,H}\circ(\beta_\sym{X}\otimes\id_H)\circ\beta_\sym{X}
&=& (\id_\sym{X}\otimes\Delta)\circ\beta_\sym{X},\\
\rho_\sym{X}\circ(\id_\sym{X}\otimes\epsilon)\circ\beta_\sym{X}
&=& \id_\sym{X}.
\end{eqnarray}
Theorem~\ref{thm_evaluationform} applies, and so
\begin{equation}
{\left<L\right>}^{(\sym{M}^H)}_{\sym{X},\id}(h)
=\phi^{(L)}(h\otimes\underbrace{\tilde\ell\otimes\cdots\otimes\tilde\ell}_{m})
\end{equation}
for all $h\in H_s$ with a linear map $\phi^{(L)}\colon H_s\otimes H^{\otimes
m}\to H_s$. The theorem computes the $S$-compatible dual quantum trace
as $\tilde\ell=T_\sym{X}\in H$, the dual quantum character of the formal comodule
$\sym{X}$. The Hennings--Kauffman--Radford evaluation is then defined by
replacing this element $\tilde\ell$ with the given $S$-compatible dual
quantum trace $\ell\in H$:
\begin{equation}
\label{eq_karaevaluation}
{\left<L\right>}^{(H)}_\ell(h) = \phi^{(L)}(h\otimes\underbrace{\ell\otimes\cdots\otimes\ell}_{m}).
\end{equation}
for all $h\in H_s$.
\end{definition}

\begin{theorem}
\label{thm_kara}
Let $H$ be a coribbon WHA, $\ell\in H$ be an $S$-compatible dual quantum
trace, and $L$ be an $m$-component (unoriented) framed link in $S^3$. The
Hennings--Kauffman--Radford evaluation ${\left<L\right>}^{(H)}_\ell$ is well
defined, \ie\ it is independent of the diagram used to represent $L$.
\end{theorem}

\begin{proof}
The proof is dual of the proof of Hennings~\cite{He96} or the proof of
Kauffman--Radford~\cite{KaRa95}. The idea is that
in~\eqref{eq_karaevaluation}, the linear map $\phi^{(L)}\colon H_s\otimes H^{\otimes
m}\to H_s$ may initially depend on the diagram that represents $L$ and on the
order of the components of $L$ used in the proof of
Theorem~\ref{thm_evaluationform}, but the linear map
\begin{equation}
\phi^{(L)}(-\otimes\underbrace{\ell\otimes\cdots\otimes\ell}_{m})\colon
H_s\to H_s
\end{equation}
is independent of the diagram that represents the link
$L$. Kauffman--Radford~\cite{KaRa95} show that the condition that $\ell\in H$
be an $S$-compatible dual quantum trace is sufficient in order to establish
independence. Since they work with left modules whereas we use
right-comodules, we can simply rotate all their diagrams by $180^\circ$ in
order to prove our claim. The coherence theorem for the monoidal category
$\sym{M}^H$ ensures that the monoidal unit $H_s\cong\one$ can be inserted at
an arbitrary position in all tensor products.

Finally, note that since $\ell$ is $S$-compatible, the dual trace $\chi=\bar
w(\ell^\prime)\ell^\pprime$ is $S$-invariant, and so the evaluation is
independent of the orientation of each individual component of $L$.
\end{proof}

The following corollary to Theorem~\ref{thm_evaluationform} establishes the
relation between the two link evaluations.

\begin{corollary}
Let $\sym{C}$ be a multi-fusion category over $k$ which has a ribbon
structure, and let $H=\coend(\sym{C},\omega)$ be the finite-dimensional and
split cosemisimple coribbon WHA reconstructed from $\sym{C}$ using the long
canonical functor~\eqref{eq_longfunctor}. Then the Reshetikhin--Turaev and the
Hennings--Kauffman--Radford evaluations agree for every (unoriented) framed
link $L$ in $S^3$:
\begin{equation}
{\left<L\right>}^{(\sym{C})}_{\hat V,\zeta}={\left<L\right>}^{(H)}_\ell
\end{equation}
with $\hat V$ of~\eqref{eq_progenerator}, $\zeta$ of~\eqref{eq_dualcentral}
and $\ell$ of~\eqref{eq_leftintegral}.
\end{corollary}

\begin{proof}
Theorem~\ref{thm_evaluationform} and Theorem~\ref{thm_integral}(5).
\end{proof}

\section{Invariants of $3$-manifolds}
\label{sect_invariant}

\subsection{The invariant for a Weak Hopf Algebra}

We proceed in analogy to the work of Hennings~\cite{He96} and
Kauffman--Radford~\cite{KaRa95} and show that if the $S$-compatible dual
quantum trace $\ell$ is a left-integral, then the link evaluation can be made
invariant under Kirby moves.

\begin{theorem}
\label{thm_invariant}
Let $H$ be a coribbon WHA over some field $k$ and $\ell\in H$ be a
left-integral which is an $S$-compatible dual quantum trace. Let $L$
be an (unoriented) framed link in $S^3$ with components
$L_1,\ldots,L_m$, $m\in\N$. If there exist $\beta,\gamma\in
k\backslash\{0\}$ such that the following two conditions,
\begin{eqnarray}
\label{eq_kirbyplus}
\nu(x^\prime)\bar\nu(\epsilon_t(x^\pprime)\ell) &=& \frac{\gamma^2}{\beta}\,\nu(x),\\
\label{eq_kirbyminus}
\bar\nu(x^\prime)\nu(\epsilon_t(x^\pprime)\ell) &=& \beta\,\bar\nu(x),
\end{eqnarray}
hold for all $x\in H$, then
\begin{equation}
\label{eq_invariant}
I(M_L) = \beta^\sigma\,\gamma^{-\sigma-m-1}\,{\left<L\right>}_\ell^{(H)}\in\End(\one)
\end{equation}
forms an invariant of connected and oriented closed (smooth)
$3$-manifolds. Here, $\one\cong H_s$ denotes the monoidal unit of $\sym{M}^H$.

In~\eqref{eq_invariant}, $M_L$ is the $3$-manifold obtained from $S^3$ by
surgery along $L$; ${\left<L\right>}_\ell^{(H)}$ denotes the generalized
Hennings--Kauffman--Radford evaluation of the link $L$ using the dual quantum
trace $\ell$; and $\sigma$ is the signature of the linking matrix of $L$ with
framing numbers on its diagonal.
\end{theorem}

\begin{proof}
By the theorems of Wallace and Lickorish~\cite{Wa60,Li62}, of
Kirby~\cite{Ki78} and of Fenn--Rourke~\cite{FeRo79}, $M_L$ is diffeomorphic as
an oriented manifold to $M_{\tilde L}$ if and only if the (unoriented) framed
link $L$ can be transformed into $\tilde L$ using a finite sequence of
Kirby-$(+1)$- and Kirby-$(-1)$-moves.

Note that $I(M_L)$ is independent of the numbering of the components of the
link, and so $I(M_L)$ is well-defined for each given Kirby diagram $L$ of some
connected and oriented closed $3$-manifold.

We first show that~\eqref{eq_kirbyplus} and~\eqref{eq_kirbyminus}
imply that ${\left<L\right>}_\ell^{(H)}$ is invariant up to the specified scalar
factors $\gamma^2/\beta$ and $\beta$ under Kirby-$(+1)$-moves and under
Kirby-$(-1)$-moves, respectively. For the Kirby-$(+1)$-move, we show that
\begin{equation}
\biggl<\tikzsymb{\kirbyplus{\sym{X}^{\otimes n}}{\sym{X}}}\biggr>_\ell^{(H)}
= \frac{\gamma^2}{\beta}\biggl<\tikzsymb{\lefttwist{\sym{X}^{\otimes n}}}\biggr>_\ell^{(H)},
\end{equation}
for all $n\in\N_0$. The left-hand side evaluates to
$(\id_\sym{X}\otimes f)\circ\beta_\sym{X}\colon\sym{X}\to\sym{X}$ with
\begin{eqnarray}
\label{eq_kirbypluscalc}
f(x)&=&\bar\nu(\ell^\prime)\bar q(x\otimes\ell^\pprime)\nn\\
&=&\nu(x^\prime)\bar\nu(x^\pprime)\bar\nu(\ell^\prime)\bar q(x^\ppprime\otimes\ell^\pprime)\nn\\
&=&\nu(x^\prime)\bar\nu(x^\pprime\ell)\nn\\
&=&\nu(x^\prime)\bar\nu(\epsilon_t(x^\pprime)\ell)\nn\\
&=&\frac{\gamma^2}{\beta}\,\nu(x),
\end{eqnarray}
for all $x\in H$ with $\bar q(-\otimes-)$ of~\eqref{eq_defbarq}. We have used
convolution invertibility of $\nu$, a consequence of
equation~\eqref{eq_ribbontensor}, that $\ell$ is a left-integral, and
equation~\eqref{eq_kirbyplus}. The result agrees with the evaluation of the
right-hand side. Recall that with our definition of a coribbon WHA $H$, the
universal ribbon form $\nu\colon H\to k$ gives rise to the isomorphisms
$\nu_X\colon X\to X$ in $\sym{M}^H$ which represent the left-handed (!)
twist. For the Kirby-$(-1)$-move, we show that
\begin{equation}
\label{eq_kirbyminuspic}
\biggl<\tikzsymb{\kirbyminus{\sym{X}^{\otimes n}}{\sym{X}}}\biggr>_\ell^{(H)}
= \beta\,\biggl<\tikzsymb{\righttwist{\sym{X}^{\otimes n}}}\biggr>_\ell^{(H)},
\end{equation}
for all $n\in\N_0$. The left-hand side evaluates to
$(\id_\sym{X}\otimes g)\circ\beta_\sym{X}\colon\sym{X}\to\sym{X}$ with
\begin{eqnarray}
g(x)&=&\nu(\ell^\prime)q(x\otimes\ell^\pprime)\nn\\
&=&\bar\nu(x^\prime)\nu(x^\pprime)\nu(\ell^\prime) q(x^\ppprime\otimes\ell^\pprime)\nn\\
&=&\bar\nu(x^\prime)\nu(x^\pprime\ell)\nn\\
&=&\bar\nu(x^\prime)\nu(\epsilon_t(x^\pprime)\ell)\nn\\
&=&\beta\,\bar\nu(x),
\end{eqnarray}
for all $x\in H$ with $q(x\otimes y)=$ of~\eqref{eq_defq}. We have used
convolution invertibility of $\nu$, equation~\eqref{eq_ribbontensor}, that
$\ell$ is a left-integral, and equation~\eqref{eq_kirbyminus}. The result
agrees with the evaluation of the right-hand side.

Since the Kirby-$(+1)$-move decreases both $m$ and $\sigma$ by one
whereas the Kirby-$(-1)$-move decreases $m$ by one and increases
$\sigma$ by one, the expression $I(M_L)$ in~\eqref{eq_invariant} is
invariant under both moves.
\end{proof}

The above argument does not simplify much if we restrict the proof to special
Kirby-$(-1)$-moves. The special Kirby-$(-1)$-move is obtained for $n=0$
in~\eqref{eq_kirbyminuspic}, \ie\ by inserting the monoidal unit for
$\sym{X}^{\otimes n}$ which is an invisible line. Note that the coribbon WHA
we use is not required to be copure, and so the invariant takes its values in
$\End(\one)$ which need not be isomorphic to $k$.

\subsection{The Hennings--Kauffman--Radford invariant}

In this section, we show that if $H$ is a finite-dimensional unimodular ribbon
Hopf algebra, then its dual $H^\ast$ satisfies the assumptions of
Theorem~\ref{thm_invariant}. In this case, our invariant~\eqref{eq_invariant}
for $H^\ast$ agrees up to a factor with the Kauffman--Radford
formulation~\cite{KaRa95} of the Hennings invariant~\cite{He96} for $H$.

First, if we work with a Hopf algebra (not weak), then
Theorem~\ref{thm_invariant} reduces to the following

\begin{corollary}
Let $H$ be a coribbon Hopf algebra over some field $k$ and $\ell\in H$
be a left-integral which is an $S$-compatible dual quantum trace. Let
$L$ be an (unoriented) framed link in $S^3$ with components
$L_1,\ldots,L_m$, $m\in\N$. If there exist $\beta,\gamma\in
k\backslash\{0\}$ such that $\bar\nu(\ell)=\gamma^2/\beta$ and
$\nu(\ell)=\beta$, then
\begin{equation}
\label{eq_invarianthopf}
I(M_L) = \beta^\sigma\,\gamma^{-\sigma-m-1}\,{\left<L\right>}^{(H)}_\ell\in k
\end{equation}
forms an invariant of connected and oriented closed (smooth) $3$-manifolds.
\end{corollary}

\begin{proof}
In a Hopf algebra, we have $\epsilon_t=\eta\circ\epsilon$. This
simplifies~\eqref{eq_kirbyplus} and~\eqref{eq_kirbyminus}. Our assumptions
$\bar\nu(\ell)=\gamma^2/\beta$ and $\nu(\ell)=\beta$ then imply these two
conditions. Finally, $\one\cong H_s=k$, and so $\End(\one)=k$.
\end{proof}

The following proposition shows that this corollary applies to the Hopf
algebra dual to the one featuring in Kauffman--Radford~\cite{KaRa95}:

\begin{proposition}
Let $H$ be a finite-dimensional unimodular ribbon Hopf algebra over some field
$k$ and $H^\ast$ be its dual Hopf algebra. Then $H^\ast$ is coribbon and has a
unique (up to a scalar) non-zero left-integral $\ell\in H^\ast$ which forms an
$S$-compatible dual quantum trace. Furthermore, the invariant $I(M_L)$
of~\eqref{eq_invarianthopf} is $\gamma$ times the invariant $\mathrm{INV}(K)$
of~\cite[page 147]{KaRa95}.
\end{proposition}

\begin{proof}
The ribbon Hopf algebra $H$ is pivotal with some pivotal element $\mu\in H$, \ie\
$\mu$ is group-like, the ribbon element is given by $r=\mu^{-1}u$ where
$u=\sum S(b_i)a_i$ is the first Drinfel'd element and $R=\sum_i a_i\otimes b_i$
denotes the universal $R$-matrix, and we have $S^2(x)=\mu x\mu^{-1}$ for all
$x\in H$. Note that in~\cite{KaRa95}, our $\mu$ and $r$ are called $G$ and $\nu$,
respectively.

By~\cite{Ra94,KaRa95}, there exists a right-cointegral $\rho\colon H\to k$,
unique up to a scalar, such that
\begin{eqnarray}
\label{eq_kara1}
\rho(xy)&=&\rho(S^2(y)x),\\
\label{eq_kara2}
\rho(\mu^2 x)&=&\rho(S(x)),
\end{eqnarray}
for all $x,y\in H$. Note that in~\cite{KaRa95}, our $\rho$ is called
$\lambda.$

We pair $H$ with its dual $H^\ast$ using the evaluation map $H^\ast\otimes
H\to k$, $\phi\otimes x\mapsto \left<\phi,x\right>=\phi(x)$. The ribbon
structure of $H$ then gives rise to a coribbon structure on $H^\ast$. In order
to match the ribbon structure of $H$ in the terminology of~\cite{KaRa95} with
our definition of a coribbon (weak) Hopf algebra, we extend the canonical
pairing to tensor products such that $\left<\phi\otimes\psi,x\otimes
y\right>=\left<\phi,y\right>\left<\psi,x\right>$ for all $\phi,\psi\in
H^\ast$ and $x,y\in H$. Note that when writing down formulas, this is the
uncommon choice, but when drawing string diagrams in $\fdVect_k$, these can
now be drawn without crossings. Also, with this choice, a left-integral of $H$
is related to a right-cointegral of $H^\ast$.

The right-cointegral $\rho\colon H\to k$ then gives rise to a left-integral
$\ell\in H^\ast$, defined by $\ell=\sum_i\rho(b_i)\beta_i$ where we have
used the canonical element $\sum_i b_i\otimes\beta_i\in H\otimes H^\ast$. The
ribbon element $r\in H$ gives rise to a universal ribbon form $\nu\colon
H^\ast\to k$ such that $\nu(\phi)=\phi(r^{-1})$ for all $\phi\in H^\ast$, and
the pivotal element $\mu\in H$ is related to the copivotal form $w\colon
H^\ast\to k$ by $w(\phi)=\phi(\mu^{-1})$ for all $\phi\in H^\ast$.

The condition~\eqref{eq_kara1} implies that $\ell$ is a dual quantum-trace,
and condition~\eqref{eq_kara2} ensures that $\ell^\prime\bar
w(\ell^\pprime)\bar w(\ell^\ppprime)=S(\ell)$ which, given that $\ell$ is a
dual quantum trace, can be shown to be equivalent to $\ell$ being
$S$-compatible.

Whereas~\cite{KaRa95} uses left-modules of $H$, we work with right-comodules
of $H^\ast$. As in the proof of Theorem~\ref{thm_kara}, our diagrams are
obtained by rotating the diagrams of~\cite{KaRa95} by $180^\circ$. Also note
that in~\cite[page 147]{KaRa95}, $\lambda(\nu)=\gamma^2/\beta$ and
$\lambda(\nu^{-1})=\beta$ which shows precisely how our invariant is related
with that of Kauffman--Radford.
\end{proof}

\subsection{The Reshetikhin--Turaev invariant}

Let $\sym{C}$ be a modular category. We now specialize our
invariant~\eqref{eq_invariant} to the case in which $H$ is the canonical WHA
obtained from $\sym{C}$ by Tannaka--Kre\v\i n reconstruction. First, we show
that the conditions~\eqref{eq_kirbyplus} and~\eqref{eq_kirbyminus} are almost
satisfied as soon as $\sym{C}$ is a multi-fusion category with a ribbon
structure.

\begin{proposition}
\label{prop_rtinvariance}
Let $\sym{C}$ be a multi-fusion category over $k$ which has a ribbon
structure, and let $H=\coend(\sym{C},\omega)$ be the finite-dimensional and
split cosemisimple coribbon WHA reconstructed from $\sym{C}$ using the long
canonical functor~\eqref{eq_longfunctor}.

Let $V\in|\sym{M}^H|$ be such that $V^\ast\cong V$ and $\zeta\colon H\to k$ be
a convolution invertible and dual central linear form that satisfies
$\zeta\circ S=\zeta$. Let finally
$\ell=w(\chi_V^\prime)\zeta(\chi_V^\pprime)\chi_V^\ppprime$ where $\chi_V$ is
the dual character of $V$. Then
\begin{eqnarray}
\label{eq_kirbypluspre}
\nu(x^\prime)\bar\nu(\epsilon_t(x^\pprime)\ell) &=& \alpha\,\nu(x),\\
\label{eq_kirbyminuspre}
\bar\nu(x^\prime)\nu(\epsilon_t(x^\pprime)\ell) &=& \beta\,\bar\nu(x),
\end{eqnarray}
for all $x\in H$ where
\begin{eqnarray}
\alpha &=& \sum_{j\in I}c_jm_j\nu_j^{-1}\dim V_j,\\
\beta  &=& \sum_{k\in I}c_jm_j\nu_j\dim V_j.
\end{eqnarray}
Here the $c_j$ are the coefficients of $\zeta$ as in
Proposition~\ref{prop_dualcentral}, the $m_j\in\N_0$ are the multiplicities of
the simple objects in $V$, \ie\ $V\cong\bigoplus_{j\in I}m_jV_j$, and the
$\nu_j$ are the eigenvalues of the ribbon twist $\nu\colon V_j\to V_j$ (the
left-handed one) on the simple objects, respectively.
\end{proposition}

\begin{proof}
First, from
\begin{equation}
\chi_V=\sum_{j\in I}m_j\sum_\ell{[e^\ell_{(V_j)}|e^{(V_j)}_\ell]}_{V_j},
\end{equation}
we compute
\begin{equation}
\ell=\sum_{j\in I}c_jm_j\sum_\ell
{[D_{\hat V}^{-1}\circ e^\ell_{(V_j)}\circ(D_{\hat V}\otimes\id_{V_j})|e^{(V_j)}_\ell]}_{V_j}.
\end{equation}
Then we evaluate the left-hand-side of~\eqref{eq_kirbypluspre} for arbitrary
$x={[\theta|v]}_X\in H$, $X\in|\sym{C}|$, $\theta\in{(\omega X)}^\ast$,
$v\in\omega X$:
\begin{gather}
(\bar\nu\circ\mu\circ(\nu\otimes\epsilon_t\otimes\id_H)\circ(\Delta\otimes\id_H))({[\theta|v]}_X\otimes\ell)\\
= \sum_{j\in I}c_jm_j\sum_{\ell,p}\nu({[\theta|e^{(X)}_p]}_X)\,\,\bar\nu
\bigl(\mu(\epsilon_t({[e^p_{(X)}|v]}_X)\otimes{[D_{\hat V}^{-1}\circ e^\ell_{(V_j)}\circ(D_{\hat V}\otimes\id_{V_j})|e^{(V_j)}_\ell]}_{V_j})\bigr).\notag
\end{gather}
We now use the expressions for $\nu$ and $\bar\nu$ of~\eqref{eq_nu}
and~\eqref{eq_nubar} and that $g_Y(\xi\otimes w)=\tr_{\hat V}(D_{\hat
V}\circ\xi\circ w)$ for all $Y\in|\sym{C}|$, $\xi\in{(\omega Y)}^\ast$,
$w\in\omega Y$, because $\sym{C}$ is spherical (\cf~\eqref{eq_bilfspherical}),
as well as $\epsilon_t$ from~\eqref{eq_epsilontnew} and obtain
\begin{gather}
\mathrm{(5.17)} = \sum_{j\in I}c_jm_j\sum_{\ell,p}
\tr_{\hat V}\biggl(D_{\hat V}\circ\theta\circ\bigl((\id_{\hat V}\otimes\nu_X)\circ e^{(X)}_p\bigr)\biggr)\\
\cdot\tr_{\hat V}\biggl(e^\ell_{(V_j)}\circ(D_{\hat V}\otimes\id_{V_j})
\circ\bigl((\Phi_X(v)\circ\Psi_X(e^p_{(X)}))\otimes\id_{V_j}\bigr)
\circ(\id_{\hat V}\otimes\nu^{-1}_{V_j})\circ e^{(V_j)}_\ell\biggr).\notag
\end{gather}
Using~\eqref{eq_dualbasispair} and the fact that the trace is cyclic and
multiplicative for tensor products of morphisms, we see that
\begin{equation}
\mathrm{(5.17)} = \sum_{j\in I}c_jm_j\sum_m
g_X\biggl(\theta\otimes\bigl((\id_{\hat V}\otimes\nu_X)\circ e^{(X)}_m\bigr)\biggl)\,
g_X(e^m_{(X)}\otimes v)\,\tr_{V_j}(\nu^{-1}_{V_j}).
\end{equation}
With the dual basis lemma and exploiting that the $V_j$ are simple, we arrive
at the right hand side of~\eqref{eq_kirbypluspre}:
\begin{equation}
\mathrm{(5.17)} = \nu({[\theta|v]}_X)\,\sum_{j\in I}c_j m_j\nu_j^{-1}\,\dim V_j.
\end{equation}
The proof of~\eqref{eq_kirbyminuspre} is identical except that $\nu_X$ and
$\nu^{-1}_X$ as well as $\nu$ and $\bar\nu$ are interchanged.
\end{proof}

\begin{corollary}
\label{cor_rtinvariance}
Let $\sym{C}$ be a multi-fusion category over $k$ which has a ribbon
structure, and let $H=\coend(\sym{C},\omega)$ be the finite-dimensional and
split cosemisimple coribbon WHA reconstructed from $\sym{C}$ using the long
canonical functor~\eqref{eq_longfunctor}. If $V=\hat V$ and $c_j=\dim V_j$,
the element $\ell$ in Proposition~\ref{prop_rtinvariance} agrees with the
left-integral~\eqref{eq_leftintegral} which is known to be an $S$-compatible
dual quantum trace. In this case, in~\eqref{eq_invariant},
\begin{equation}
{\left<-\right>}_\ell^{(H)}={\left<-\right>}^{(\sym{C})}_{\hat V,\zeta}
\end{equation}
is the Reshetikhin--Turaev evaluation. Furthermore, we compute that
\begin{equation}
\alpha=\sum_{j\in I}\nu_j^{-1}{(\dim V_j)}^2,\qquad
\beta=\sum_{j\in I}\nu_j{(\dim V_j)}^2.
\end{equation}
\end{corollary}
Finally, the following proposition shows that if $\sym{C}$ is modular, then
$\alpha\neq0$ and $\beta\neq0$.

\begin{proposition}
Let $\sym{C}$ be a modular category, linear over the field $k$, and let
$H=\coend(\sym{C},\omega)$ be the finite-dimensional and split cosemisimple
coribbon WHA reconstructed from $\sym{C}$ using the long canonical
functor~\eqref{eq_longfunctor}. Assume that in
Proposition~\ref{prop_rtinvariance}, we have $V=\hat V$ and $c_j=\dim
V_j$. Then $\alpha\neq 0$ and $\beta\neq 0$ in that proposition.
\end{proposition}

\begin{proof}
Let $i,j\in I$, and denote by $S_{ij}$ the coefficients of the
$S$-matrix. Then
\begin{eqnarray}
\alpha\,S_{ij}
&=&\alpha\,{\biggl<\tikzsymb{\hopflink{i}{j}}\biggr>}_\ell^{(H)}
=\sum_{p\in I}\dim V_p\,{\biggl<\tikzsymb{\triplelink{i}{p}{j}}\biggr>}_\ell^{(H)}\notag\\
&=&\sum_{p\in I}\nu_i^{-1}\nu_p^{-1}\nu_j^{-1} S_{ip}S_{pj},
\end{eqnarray}
where we have used that the $V_i$, $i\in I$ are simple; a Kirby-(+1)-move; and
again that the $V_i$ are simple. This implies that
\begin{equation}
\alpha S_{ij}\nu_i\nu_j
= \sum_{p\in I}\nu_p^{-1} S_{ip} S_{pj}.
\end{equation}
If $\alpha=0$, then the right hand side would vanish for all $i,j\in I$,
\ie\ the matrix product $ST=0$ where $T_{pj}=\nu^{-1}_p S_{pj}$. This
contradicts the invertibility of $S$. Interchanging $\nu$ and $\nu^{-1}$ as
well as $\alpha$ and $\beta$ in the above argument establishes that
$\beta=0$.
\end{proof}

\begin{corollary}
Let $\sym{C}$ be a modular category, linear over the field $k$, such that its
global dimension has a square root $\sym{D}\in k$, \ie\
\begin{equation}
\sym{D}^2 = \sum_{j\in I}{(\dim V_j)}^2.
\end{equation}
Let $H=\coend(\sym{C},\omega)$ be the finite-dimensional and split
cosemisimple coribbon WHA reconstructed from $\sym{C}$ using the long
canonical functor~\eqref{eq_longfunctor}. In
Proposition~\ref{prop_rtinvariance}, assume that $V=\hat V$ and $c_j=\dim
V_j$. Then~\eqref{eq_kirbyplus} and~\eqref{eq_kirbyminus} hold for
$\gamma=\sym{D}$, and the invariant~\eqref{eq_invariant} agrees with the
Reshetikhin--Turaev invariant.
\end{corollary}

\begin{proof}
Recall from Corollary~\ref{cor_rtinvariance} that with these choices of $V$
and $c_j$, ${\left<-\right>}_\ell^{(H)}={\left<-\right>}^{(\sym{C})}_{\hat V,\zeta}$.
In~\cite[Section II.3]{Tu10}, our $\beta$ is called $\Delta$, and our $\alpha$
corresponds to $d_0^{-1}$ (recall that our $\nu$ is the left-handed ribbon
twist, but the $\nu$ in~\cite{Tu10} is the right-handed one). Furthermore,
in~\cite[Section II.3]{Tu10}, it is shown that
$\alpha\beta=\Delta/d_0=\sym{D}^2$, \ie\ with our choice of $\gamma=\sym{D}$,
we arrive at $\alpha=\gamma^2/\beta$ as required in~\eqref{eq_kirbyplus}.
\end{proof}

The preceding corollary shows in particular that our proofs of
Theorem~\ref{thm_invariant} and Proposition~\ref{prop_rtinvariance} can be
combined to a new proof of invariance of the Reshetikhin--Turaev invariant,
\cf~\cite[Section II.3]{Tu10}.

\subsection{Contrast with the Lyubashenko invariant}

Let us finally point out what is the difference between our invariant $I(M_L)$
(Theorem~\ref{thm_invariant}) for the WHA $H=\coend(\sym{C},\omega)$
reconstructed from a modular category $\sym{C}$ and Lyubashenko's
invariant~\cite{Ly95} that uses Majid's coend
$F=\coend(\sym{C},1_\sym{C})$. Although these two invariants take the same
numerical value whenever each of them agrees with the Reshetikhin--Turaev
invariant, their computation is rather different.

Whereas our $H$ is a vector space equiped with linear structure maps that make
it a WHA, Lyubashenko's $F$ is a Hopf algebra object $H\in|\sym{C}|$, \ie\ an
object of $\sym{C}$ whose structure maps are morphisms in $\sym{C}$.

Let, for example, $\sym{C}_3$ be a modular category associated with
$U_q(\mathrm{sl}_2)$ with $3$ simple objects up to isomorphism. We denote
representatives of the classes of simple objects by $X_0\cong\one$, $X_1$ and
$X_2$. The fusion rules are  given by $X_1\otimes X_1\cong X_0\oplus X_2$,
$X_1\otimes X_2\cong X_1$ and $X_2\otimes X_2\cong X_0$. We know that
$\sym{C}_3\simeq\sym{M}^H$ for our reconstructed WHA.

Lyubashenko's coend is the $H$-comodule
\begin{equation*}
F\cong (X_0\otimes X_0^\ast) \oplus (X_1\otimes X_1^\ast) \oplus (X_2\otimes X_2^\ast)
\cong 3X_0 \oplus X_2,
\end{equation*}
where we write $nX=X\oplus\cdots\oplus X$ (direct sum of $n$ terms). Note that
its decomposition does not contain $X_1$ and that the $X_1$-term is included
as $X_1\otimes X_1^\ast\hookrightarrow X_0\oplus X_2$ (as $H$-comodules). My
coend, however, is $H$ itself with the regular coaction, \ie\ the following
$H$-comodule:
\begin{equation*}
H\cong 3X_0 \oplus 4X_1 \oplus 3X_2.
\end{equation*}
Therefore, obviously, the two coends differ as objects of $\sym{M}^H$ and also
have different $k$-dimensions. Note that $\dim_k X_0=3$, $\dim_k X_1=4$ and
$\dim_k X_2=3$ in $\sym{C}_3\simeq\sym{M}^H$.

\appendix
\section{Weak Hopf Algebras and their corepresentations}
\label{app_wha}

In this appendix, we collect the relevant definitions on WHAs with
additional structure.

\subsection{Weak Bialgebras}
\label{app_wba}

Given a WHA $H$ over some field $k$, the source and target
\emph{counital maps} are given by
\begin{eqnarray}
\epsilon_s &:=&
(\id_H\otimes\epsilon)\circ(\id_H\otimes\mu)\circ(\tau_{H,H}\otimes\id_H)
\circ(\id_H\otimes\Delta)\circ(\id_H\otimes\eta)\colon H\to H,\\
\epsilon_t &:=&
(\epsilon\otimes\id_H)\circ(\mu\otimes\id_H)\circ(\id_H\otimes\tau_{H,H})
\circ(\Delta\otimes\id_H)\circ(\eta\otimes\id_H)\colon H\to H.
\end{eqnarray}
Here $\tau_{V,W}(v\otimes w)=w\otimes v$, $v\in V$, $w\in W$, denotes
the symmetric braiding in $\Vect_k$. The mutually commuting
subalgebras $H_s:=\epsilon_s(H)$ and $H_t:=\epsilon_t(H)$ are called
the source and target \emph{base algebra} of $H$, respectively.

A WBA [WHA] is a bialgebra [Hopf algebra] if and only if
$\epsilon_s=\eta\circ\epsilon$, if and only if
$\epsilon_t=\eta\circ\epsilon$, if and only if $H_s\cong k$, and if
and only if $H_t\cong k$.

If $H$ is a finite-dimensional WBA [WHA], then so is its dual vector
space $H^\ast$. Every finite-dimensional WHA $H$ has an invertible
antipode.

We abbreviate $\eta(1)=1$ and use Sweedler's notation
$\Delta(x)=x^\prime\otimes x^\pprime$ for the comultiplication of
$x\in H$ and $\beta(v)=v_0\otimes v_1$ for the coaction $\beta\colon
V\to V\otimes H$ of a right $H$-comodule $V$. If $H$ is
finite-dimensional, we also use the Sweedler arrows
$x\rightharpoonup\phi=\phi^\prime(x)\phi^\pprime$ and
$\phi\rightharpoonup x=x^\prime\phi(x^\pprime)$ for $x\in H$, $\phi\in
H^\ast$.

The category $\sym{M}^H$ of finite-dimensional right $H$-comodules of a WBA
forms a $k$-linear abelian monoidal category
$(\sym{M}^H,\otimes,\one,\alpha,\lambda,\rho)$ as follows~\cite{Pf09a}. The
source base algebra $H_s$ is a right $H$-comodule with
\begin{equation}
\beta_{H_s}\colon H_s\to H_s\otimes H,\qquad
x\mapsto x^\prime\otimes x^\pprime.
\end{equation}
It forms the monoidal unit object, $\one=H_s$ of $\sym{M}^H$.
Given $V,W\in|\sym{M}^H|$, their tensor product is the vector space
\begin{equation}
V\otimes W := \{\,v\otimes w\in V\otimes W\mid\quad v\otimes w= (v_0\otimes w_0)\epsilon(v_1w_1)\,\}
\end{equation}
with the coaction
\begin{equation}
\beta_{V\otimes W}\colon V\otimes W\to (V\otimes W)\otimes H,\quad
v\otimes w \mapsto (v_0\otimes w_0)\otimes (v_1w_1).
\end{equation}
The associator $\alpha_{U,V,W}\colon(U\otimes V)\otimes W\to
U\otimes(V\otimes W)$ is induced from that of $\Vect_k$. The left
and right unit constraints are given by
\begin{alignat}{3}
\lambda_V &\colon H_s\otimes V\to V,&&\quad h\otimes v\mapsto v_0\epsilon(hv_1),\\
\rho_V    &\colon V\otimes H_s\to V,&&\quad v\otimes h\mapsto v_0\epsilon(v_1h).
\end{alignat}
For convenience, we list their inverses, too:
\begin{alignat}{3}
\lambda^{-1}_V &\colon V\to H_s\otimes V,&&\quad v\mapsto (1^\prime\otimes v_0)\epsilon(1^\pprime v_1),\\
\rho^{-1}_V    &\colon V\to V\otimes H_s,&&\quad v\mapsto v_0\otimes\epsilon_s(v_1).
\end{alignat}

\subsection{Weak Hopf Algebras}

A \emph{left-autonomous}~\cite{FrYe92} category $\sym{C}$ is a
monoidal category in which every object $X\in|\sym{C}|$ is equipped
with a \emph{left-dual} $(X^\ast,\ev_X,\coev_X)$, \ie\ an object
$X^\ast\in|\sym{C}|$ and morphisms $\ev_X\colon X^\ast\otimes
X\to\one$ (\emph{left evaluation}) and $\coev_X\colon\one\to X\otimes
X^\ast$ (\emph{left coevaluation}) that satisfy the triangle
identities
\begin{eqnarray}
\rho_X\circ(\id_X\otimes\ev_X)\circ\alpha_{X,X^\ast,X}
\circ(\coev_X\otimes\id_X)\circ\lambda_X^{-1} &=& \id_X,\\
\lambda_{X^\ast}\circ(\ev_X\otimes\id_{X^\ast})\circ\alpha^{-1}_{X^\ast,X,X^\ast}
\circ(\id_{X^\ast}\otimes\coev_X)\circ\rho_{X^\ast}^{-1} &=& \id_{X^\ast}.
\end{eqnarray}
In a left-autonomous category $\sym{C}$, the \emph{left-dual} of a
morphism $f\colon X\to Y$ is defined as
\begin{equation}
f^\ast := \lambda_{X^\ast}\circ(\ev_Y\otimes\id_{X^\ast})\circ\alpha^{-1}_{Y^\ast,Y,X^\ast}
\circ(\id_{Y^\ast}\otimes(f\otimes\id_{X^\ast}))\circ(\id_{Y^\ast}\otimes\coev_X)\circ\rho^{-1}_{Y^\ast}.
\end{equation}

Let $\sym{C}$ be a left-autonomous category. The functor
$\ast\colon\sym{C}\to\sym{C}^\op$ that sends each object and each
morphism to their left-dual, is strong monoidal.

Let $H$ be a WHA. The category $\sym{M}^H$ is \emph{left-autonomous}
as follows. For each $V\in|\sym{M}^H|$, the dual vector space $V^\ast$
forms a right $H$-comodule with
\begin{equation}
\label{eq_dualaction}
\beta_{V^\ast}\colon V^\ast\to V^\ast\otimes H,\qquad
\theta\mapsto (v\mapsto \theta(v_0)\otimes S(v_1)).
\end{equation}
The left-dual of $V$ is given by $(V^\ast,\ev_V,\coev_V)$ where
\begin{alignat}{3}
\label{eq_ev}
\ev_V   &\colon V^\ast\otimes V\to H_s,&&\quad \theta\otimes v\to\theta(v_0)\epsilon_s(v_1),\\
\label{eq_coev}
\coev_V &\colon H_s\to V\otimes V^\ast,&&\quad x\mapsto\sum_j ({(e_j)}_0\otimes e^j)\epsilon(x{(e_j)}_1).
\end{alignat}
Here we have used the evaluation and coevaluation maps that turn
$V^\ast$ into a left-dual of $V$ in $\fdVect_k$:
\begin{alignat}{3}
\ev_V^{(\fdVect_k)}   &\colon V^\ast\otimes V\to k,&&\quad \theta\otimes v\mapsto \theta(v),\\
\coev_V^{(\fdVect_k)} &\colon k\to V\otimes V^\ast,&&\quad 1\mapsto \sum_j e_j\otimes e^j.
\end{alignat}

\subsection{Copivotal Weak Hopf Algebras}

Let $H$ be a WBA. A linear form $f\colon H\to k$ is said to be
\emph{convolution invertible} if there exists some linear form $\bar
f\colon H\to k$ such that $f(x^\prime)\bar
f(x^\pprime)=\epsilon(x)=\bar f(x^\prime)f(x^\pprime)$ for all $x\in
H$. The linear form $f$ is called \emph{dual central} if
$f(x^\prime)x^\pprime = x^\prime f(x^\pprime)$ for all $x\in H$. It is
called \emph{dual group-like} if $\epsilon(x^\prime
y^\prime)f(x^\pprime)f(y^\pprime) = f(xy) =
f(x^\prime)f(y^\prime)\epsilon(x^\pprime y^\pprime)$ for all $x,y\in
H$ and $f(\epsilon_t(x))=\epsilon(x)=f(\epsilon_s(x))$ for all $x\in
H$. Note that in a WHA, every dual group-like linear form is
convolution invertible with $\bar f(x)=f(S(x))$.

A WHA $H$ is called \emph{copivotal}~\cite{Pf09b} if there exists a
dual group-like linear form $w\colon H\to k$ such that $S^2(x) =
w(x^\prime)x^\pprime\bar w(x^\ppprime)$ for all $x\in H$.

A \emph{pivotal} category~\cite{FrYe92} $\sym{C}$ is a left-autonomous
category with a monoidal natural equivalence $\tau\colon
1_\sym{C}\Rightarrow \ast\circ\ast$ such that ${(\tau_X)}^\ast =
\tau_{X^\ast}^{-1}$ for all $X\in|\sym{C}|$.

Given a copivotal WHA $H$, the category $\sym{M}^H$ is pivotal with
$\tau_V\colon V\to {V^\ast}^\ast$ given by
\begin{equation}
\tau_V(v) = \tau_V^{(\fdVect_k)}(v_0)w(v_1)
\end{equation}
for all $V\in|\sym{M}^H|$ and $v\in V$. Here we denote by
$\tau_V^{(\fdVect_k)}\colon V\to {V^\ast}^\ast$ the pivotal structure of
$\fdVect_k$ which is just the usual canonical identification $V\cong
{V^\ast}^\ast$.

A \emph{right-autonomous} category $\sym{C}$ is a monoidal category in
which every object $X\in|\sym{C}|$ is equipped with a \emph{right-dual}
$(\bar X,\bar\ev_X,\bar\coev_X)$, \ie\ an object $\bar X\in|\sym{C}|$
with morphisms $\bar\ev_X\colon X\otimes\bar X\to\one$ (\emph{right
evaluation}) and $\bar\coev_X\colon\one\to\bar X\otimes X$
(\emph{right coevaluation}) that satisfy the triangle identities
\begin{eqnarray}
\lambda_X\circ(\bar\ev_X\otimes\id_X)\circ\alpha^{-1}_{X,\bar X,X}
\circ(\id_X\otimes\bar\coev_X)\circ\rho_X^{-1} &=& \id_X,\\
\rho_{\bar X}\circ(\id_{\bar X}\otimes\bar\ev_X)\circ\alpha_{\bar X,X,\bar X}
\circ(\bar\coev_X\otimes\id_{\bar X})\circ\lambda^{-1}_{\bar X} &=& \id_{\bar X}.
\end{eqnarray}

Note that every pivotal category $\sym{C}$ is not only left-, but also
right-autonomous with $\bar X=X^\ast$ and
\begin{eqnarray}
\label{eq_barev}
\bar\ev_X   &=& \ev_{X^\ast}\circ(\tau_X\otimes\id_{X^\ast}),\\
\label{eq_barcoev}
\bar\coev_X &=& (\id_{X^\ast}\otimes\tau_X^{-1})\circ\coev_{X^\ast}
\end{eqnarray}
for all $X\in|\sym{C}|$. We can therefore define the \emph{right-dual}
of a morphism $f\colon X\to Y$ as
\begin{equation}
\bar f:=\rho_{\bar X}\circ(\id_{\bar X}\otimes\bar\ev_Y)\circ\alpha_{\bar X,Y,\bar Y}
\circ((\id_{\bar X}\otimes f)\otimes\id_{\bar Y})\circ(\bar\coev_X\otimes\id_{\bar Y})\circ\lambda^{-1}_{\bar Y}.
\end{equation}
It can be shown to agree with the left-dual, \ie\ $\bar f=f^\ast$.

Using both left- and right-duals, we can define two traces of a morphism $f\colon X\to Y$, the
\emph{left-trace}
\begin{equation}
\tr^{(L)}_X(f) = \ev_X\circ(\id_{X^\ast}\otimes f)\circ\bar\coev_X\colon\one\to\one
\end{equation}
and the \emph{right-trace}
\begin{equation}
\tr^{(R)}_X(f) = \bar\ev_X\circ(f\otimes\id_{X^\ast})\circ\coev_X\colon\one\to\one.
\end{equation}
Both left- and right-traces are cyclic, \ie\ $\tr^{(L)}_X(g\circ
f)=\tr^{(L)}_Y(f\circ g)$ for all $f\colon X\to Y$ and $g\colon Y\to
X$ and similarly for the right-trace. In general, however, left- and
right-traces need not agree.

\subsection{Cospherical Weak Hopf Algebras}

A \emph{spherical category}~\cite{BaWe99} is a pivotal category in
which $\tr^{(L)}_X(f) = \tr^{(R)}_X(f)$ for all morphisms $f\colon
X\to X$ in $\sym{C}$. In this case, the above expression is just
called the \emph{trace} of $f$ and denoted by $\tr_X(f)$, and $\dim(X)
= \tr_X(\id_X)$ is called the \emph{dimension} of $X$. Note that in a
spherical category, $\tr_X(f)=\tr_{X^\ast}(f^\ast)$ for every morphism
$f\colon X\to X$ and thus $\dim(X)=\dim(X^\ast)$. Finally,
$\tr_{X_1\otimes X_2}(h_1\otimes h_2)=\tr_{X_1}(h_1)\tr_{X_2}(h_2)$
for all $h_j\colon X_j\to X_j$, $j\in\{1,2\}$.

A \emph{cospherical} WHA~\cite{Pf09b} $H$ is a copivotal WHA for which
$\tr^{(L)}_V(f) = \tr^{(R)}_V(f)$ for all morphisms $f\colon V\to V$
of $\sym{M}^H$. If $H$ is a cospherical WHA, then $\sym{M}^H$ is
therefore spherical.

\subsection{Coquasitriangular Weak Hopf Algebras}

A \emph{coquasitriangular} WHA~\cite{Pf09a} is a WHA with a linear
form $r\colon H\otimes H\to k$, the \emph{universal $r$-form}, that
satisfies the following conditions:
\begin{myenumerate}
\item
For all $x,y\in H$,
\begin{equation}
\label{eq_coquasidef}
\epsilon(x^\prime y^\prime)r(x^\pprime\otimes y^\pprime)=r(x\otimes y)
=r(x^\prime\otimes y^\prime)\epsilon(y^\pprime x^\pprime).
\end{equation}
\item
There exists some linear $\bar r\colon H\otimes H\to k$ that is a
weak convolution inverse of $r$, \ie\
\begin{eqnarray}
\label{eq_coquasiinv1}
\bar r(x^\prime\otimes y^\prime)r(x^\pprime\otimes y^\pprime)&=&\epsilon(yx),\\
\label{eq_coquasiinv2}
r(x^\prime\otimes y^\prime)\bar r(x^\pprime\otimes y^\pprime)&=&\epsilon(xy).
\end{eqnarray}
\item
For all $x,y,z\in H$,
\begin{eqnarray}
x^\prime y^\prime r(x^\pprime\otimes y^\pprime)
&=&r(x^\prime\otimes y^\prime)y^\pprime x^\pprime,\\
r((xy)\otimes z)&=&r(y\otimes z^\prime) r(x\otimes z^\pprime),\\
r(x\otimes (yz))&=&r(x^\prime\otimes y) r(x^\pprime\otimes z).
\end{eqnarray}
\end{myenumerate}
Note that $\bar r$ in~(2) is uniquely determined by $r$ if one
imposes~\eqref{eq_coquasidef}, \eqref{eq_coquasiinv1}
and~\eqref{eq_coquasiinv2}.

In a coquasitriangular WHA $H$, we define the linear form $q\colon H\otimes H\to k$ by
\begin{equation}
\label{eq_defq}
q(x\otimes y) = r(x^\prime\otimes y^\prime)r(y^\pprime\otimes x^\pprime),
\end{equation}
for all $x,y\in H$. Its weak convolution inverse $\bar q\colon
H\otimes H\to k$ is then given by
\begin{equation}
\label{eq_defbarq}
\bar q(x\otimes y) = \bar r(y^\prime\otimes x^\prime)\bar r(x^\pprime\otimes y^\pprime),
\end{equation}
for all $x,y\in H$. The \emph{dual Drinfel'd elements} are the linear
forms $u\colon H\to k$ and $v\colon H\to k$ given by
$u(x)=r(S(x^\pprime)\otimes x^\prime)$ and $v(x)=r(S(x^\prime)\otimes
x^\pprime)$ for all $x\in H$.

A \emph{braided monoidal category} $\sym{C}$ is a monoidal category
with natural isomorphisms $\sigma_{X,Y}\colon X\otimes Y\to Y\otimes
X$ for all $X,Y\in|\sym{C}|$ that satisfy the two hexagon axioms
\begin{eqnarray}
\sigma_{X\otimes Y,Z} &=& \alpha_{Z,X,Y}\circ(\sigma_{X,Z}\otimes\id_Y)
\circ\alpha^{-1}_{X,Z,Y}\circ(\id_X\otimes\sigma_{Y,Z})\circ\alpha_{X,Y,Z},\\
\sigma_{X,Y\otimes Z} &=& \alpha^{-1}_{Y,Z,X}\circ(\id_Y\otimes\sigma_{X,Z})
\circ\alpha_{Y,X,Z}\circ(\sigma_{X,Y}\otimes\id_Z)\circ\alpha^{-1}_{X,Y,Z}
\end{eqnarray}
for all $X,Y,Z\in|\sym{C}|$.

If $H$ is a coquasitriangular WHA, then the category $\sym{M}^H$ is
braided monoidal with braiding
$\sigma_{V,W}\colon V\otimes W\to W\otimes V$ given by
\begin{equation}
\label{eq_braiding}
\sigma_{V,W}(v\otimes w)=(w_0\otimes v_0)r(w_1\otimes v_1)
\end{equation}
for all $V,W\in|\sym{M}^H|$ and $v\in V$, $w\in W$. Note that
\begin{equation}
\label{eq_defQ}
Q_{V,W}=\sigma_{W,V}\circ\sigma_{V,W}
\end{equation}
can be computed as $Q_{V,W}(v\otimes w)=(v_0\otimes w_0)q(v_1\otimes w_1)$ for
all $v\in V$, $w\in W$, and similarly $Q_{V,W}^{-1}(v\otimes w)=(v_0\otimes
w_0)\bar q(v_1\otimes w_1)$.

\subsection{Coribbon Weak Hopf Algebras}
\label{app_coribbon}

A \emph{coribbon} WHA~\cite{Pf09a} is a coquasitriangular WHA with a
convolution invertible and dual central linear form $\nu\colon H\to
k$, the \emph{universal ribbon twist}, such that
\begin{eqnarray}
\label{eq_ribbontensor}
\nu(xy)&=&\nu(x^\prime)\nu(y^\prime)q(x^\pprime\otimes y^\pprime),\\
\nu(S(x))&=&\nu(x)
\end{eqnarray}
all $x,y\in H$.

A \emph{ribbon category} is a braided monoidal category that is
left-autonomous with natural isomorphisms $\nu_X\colon X\to X$, the
\emph{ribbon twist}, such that
\begin{equation}
\nu_{X\otimes Y} = \sigma_{Y,X}\circ\sigma_{X,Y}\circ(\nu_X\otimes\nu_Y)
\end{equation}
and
\begin{equation}
(\nu_X\otimes\id_{X^\ast})\circ\coev_X = (\id_X\otimes\nu_{X^\ast})\circ\coev_X
\end{equation}
for all $X,Y\in|\sym{C}|$.

Note that every ribbon category is pivotal with $\tau_X\colon X\to
{X^\ast}^\ast$ given by
\begin{equation}
\tau_X = \lambda_{{X^\ast}^\ast}\circ(\ev_X\otimes\id_{{X^\ast}^\ast})
\circ(\sigma_{X,X^\ast}\otimes\id_{{X^\ast}^\ast})
\circ(\nu_X\otimes\coev_{X^\ast})\circ\rho_X^{-1},
\end{equation}
and furthermore spherical. For convenience, we give the right
evaluation and coevaluation of~\eqref{eq_barev}
and~\eqref{eq_barcoev}:
\begin{eqnarray}
\label{eq_barevribbon}
\bar\ev_X   &=& \ev_X\circ\sigma_{X,X^\ast}\circ(\nu_X\otimes\id_{X^\ast}),\\
\label{eq_barcoevribbon}
\bar\coev_X &=& (\id_{X^\ast}\otimes\nu_X)\circ\sigma_{X,X^\ast}\circ\coev_X.
\end{eqnarray}

If $H$ is a coribbon WHA, then $\sym{M}^H$ is a ribbon category with
ribbon twist
\begin{equation}
\label{eq_twist}
\nu_V\colon V\to V,\qquad v\mapsto v_0\nu(v_1),
\end{equation}
for all $V\in|\sym{M}^H|$ and $v\in V$. Every coribbon WHA is
cospherical with the copivotal form $w(x)=v(x^\prime)\nu(x^\pprime)$
for all $x\in H$, involving the second Drinfel'd element and the
universal ribbon form.

Using the definition of a ribbon category as in this Appendix, we can
draw the corresponding string diagrams. If we draw composition from
top to bottom and the tensor product from left to right, we arrive at
the diagrams shown in Section~\ref{sect_evaluation}.

\subsection{Cosemisimple Weak Hopf Algebras and fusion categories}
\label{sect_fusion}

A monoidal category $\sym{C}$ is called $k$-linear over some field $k$
if the underlying category is $k$-linear, \ie\ enriched in the
category $\Vect_k$, and the tensor product of morphisms is
$k$-bilinear. A $k$-linear category is called \emph{additive} if it
has a terminal object and all binary products. Such a category
automatically has all finite biproducts. A $k$-linear category is
\emph{abelian} if it is additive, has all finite limits, and if every
monomorphism is a kernel and every epimorphism a cokernel. Note that
in $k$-linear pivotal categories, the traces
$\tr_X^{(L)},\tr_X^{(R)}\colon\End(X)\to\End(\one)$ are $k$-linear.

In a $k$-linear additive category $\sym{C}$, we call an object
$X\in|\sym{C}|$ \emph{simple} if $\End(X)\cong k$. A $k$-linear
additive monoidal category is called \emph{pure} if the monoidal unit
$\one$ is simple. A $k$-linear additive category is called \emph{split
semisimple} if every object $X\in|\sym{C}|$ is isomorphic to a
finite biproduct of simple objects. It is called \emph{finitely split
semisimple} if it is split semisimple and there exist only a finite
number of simple objects up to isomorphism. If $\sym{C}$ is a
$k$-linear additive category that is split semisimple, we denote by
${\{V_j\}}_{j\in I}$ a family of representatives of the isomorphism
classes of the simple objects $V_j$, $j\in I$, of $\sym{C}$, indexed
by the set $I$.

A \emph{multi-fusion category} $\sym{C}$ over $k$, see, for
example~\cite{EtNi05}, is a finitely split semisimple $k$-linear
additive autonomous monoidal category such that $\Hom(X,Y)$ is
finite-dimensional over $k$ for all $X,Y\in|\sym{C}|$. A \emph{fusion
category} over $k$ is a pure multi-fusion category. Note that every
multi-fusion category is abelian and essentially small, and that in
every multi-fusion category, if $X\in|\sym{C}|$ is simple, then so is
$X^\ast$.

If $C$ is a coalgebra and $V$ a finite-dimensional right $C$-comodule
with coaction $\beta\colon V\to V\otimes C$ and basis ${(v_j)}_j$,
then there are elements $c^{(V)}_{\ell j}\in C$ uniquely determined by
the condition that $\beta_V(v_j)=\sum_\ell v_\ell\otimes c^{(V)}_{\ell
j}$. They are called the \emph{coefficients of} $V$ with respect to
that basis. They span the \emph{coefficient coalgebra}
$C(V)=\Span_k\{c^{(V)}_{\ell j}\}$, a subcoalgebra of $C$.

Let $W$ be a finite-dimensional vector space over $k$ with dual space
$W^\ast$ and a pair of dual bases ${(e_j)}_j$ and ${(e^j)}_j$ of $W$
and $W^\ast$, respectively. We abbreviate $c^{(W)}_{jk}=e^j\otimes
e_k\in W^\ast\otimes W$. The coalgebra $(W^\ast\otimes
W,\Delta,\epsilon)$ with $\Delta(c^{(W)}_{jk})=\sum_\ell
c^{(W)}_{j\ell}\otimes c^{(W)}_{\ell k}$ and
$\epsilon(c^{(W)}_{jk})=\delta_{jk}$ is called the \emph{matrix
coalgebra} associated with $W$. In this case, $W$ is a right
$W^\ast\otimes W$-comodule, and $W^\ast\otimes W$ is its coefficient
coalgebra.

A coalgebra $C$ is called \emph{cosimple} if $C$ has no subcoalgebras
other than $C$ and $\{0\}$. The coalgebra $C$ is called
\emph{cosemisimple} if it is a coproduct in $\Vect_k$ of cosimple
coalgebras. The coalgebra $C$ is called \emph{split cosemisimple} if
it is cosemisimple and every cosimple subcoalgebra is a matrix
coalgebra. A right $C$-comodule $V$ of some coalgebra $C$ is called
\emph{irreducible} if $V\neq\{0\}$ and $V$ has no subcomodules other
than $V$ and $\{0\}$.

If $H$ is a WHA over some field $k$, then $\sym{M}^H$ is a $k$-linear
abelian autonomous monoidal category such that $\Hom(X,Y)$ is
finite-dimensional over $k$ for all $X,Y\in|\sym{C}|$. If $H$ is in
addition [finite-dimensional and] split cosemisimple, then $\sym{M}^H$
is [finitely] split semisimple. A WHA is called \emph{copure} if its
base algebras intersect trivially, \ie\ if $H_s\cap H_t\cong k$, see,
for example~\cite{Pf09a}. In this case, $\sym{M}^H$ is pure.

A WHA over $k$ is said to be \emph{multi-fusion} if it is
finite-dimensional and split cosemisimple. It is called \emph{fusion}
if it is in addition copure. In these cases, $\sym{M}^H$ is multi-fusion
or fusion over $k$, respectively.

\subsection{Comodular Weak Hopf Algebras}

Let $\sym{C}$ be a ribbon category and $V,W\in|\sym{C}|$. We call the
evaluation of the Hopf link with components labeled by $V$ and $W$,
\begin{equation}
S_{V,W} := \tr_{V\otimes W}(Q_{V,W})\in\End(\one).
\end{equation}
If $\sym{C}$ is in addition multi-fusion with a family ${\{V_j\}}_{j\in I}$ of
representatives of the isomorphism classes of objects, we write
$S_{j\ell}:=S_{V_jV_\ell}$, $j,\ell\in I$. If $\sym{C}$ is fusion,
$\End(\one)\cong k$, and so $S_{j\ell}\in k$. A \emph{modular category} is a
fusion category that has the structure of a ribbon category and for which the
$|I|\times|I|$-matrix with coefficients $S_{j\ell}$, $j,\ell\in I$, is
non-degenerate.

Let $H$ be a copure coribbon WHA with copivotal form $w\colon H\to k$ and
$V\in|\sym{M}^H|$, $n:=\dim_kV$. The \emph{dual character} of $V$ is the
element
\begin{equation}
\label{eq_dualchar}
\chi_V=\sum_{j=1}^n c^{(V)}_{jj}\in H,
\end{equation}
and the \emph{dual quantum character} the element
\begin{equation}
\label{eq_dualqchar}
T_V=\sum_{j,\ell=1}^n c^{(V)}_{j\ell}w(c^{(V)}_{\ell j})\in H.
\end{equation}
We denote the space of dual quantum characters of $H$ by
\begin{equation}
T(H) = \Span_k\{\,T_V\colon\quad V\in|\sym{M}^H|\,\}.
\end{equation}

If $H$ is a copure coribbon WHA and the morphism $f^{(\gamma)}_{V,W}\colon
V\otimes W\to V\otimes W$ is of the form
\begin{equation}
f^{(\gamma)}_{V,W}
=(\id_{V\otimes W}\otimes\gamma)\circ(\id_V\otimes\tau_{H,W}\otimes\id_H)
\circ(\beta_V\otimes\beta_W)
\end{equation}
with a linear form $\gamma\colon H\otimes H\to k$ that satisfies
\begin{eqnarray}
x^\prime y^\prime\gamma(x^\pprime\otimes y^\pprime)
&=& \gamma(x^\prime\otimes y^\prime)x^\pprime y^\pprime,\\
\epsilon(x^\prime y^\prime)\gamma(x^\pprime\otimes y^\pprime)
&=& \gamma(x\otimes y)
\end{eqnarray}
for all $x,y\in H$, then
\begin{equation}
\tr_{V\otimes W}(f^{(\gamma)}_{V,W}) = c^{(\gamma)}_{V,W}\,\id_{H_s},
\end{equation}
where the element $c^{(\gamma)}_{V,W}\in k$ is determined by
\begin{equation}
\gamma(T_V^\prime\otimes T_W^\prime)\epsilon_s(S(T_V^\pprime T_W^\pprime))
= c^{(\gamma)}_{V,W}\eta(1).
\end{equation}

Given a copure coribbon WHA, we can therefore define a linear form $\tilde
q\colon T(H)\otimes T(H)\to k$, $T_V\otimes T_W\to \tilde q_{V,W}$ where the
$\tilde q_{V,W}\in k$ are determined by
\begin{equation}
q(T_V^\prime\otimes T_W^\prime)\epsilon_s(S(T_V^\pprime T_W^\pprime))
= \tilde q_{V,W}\,\eta(1).
\end{equation}
$H$ is called \emph{weakly cofactorizable} if every linear form $\phi\colon
T(H)\to k$ can be written as $\phi(-)=\tilde q(-\otimes x)$ for some $x\in
T(H)$.

A \emph{comodular} WHA~\cite{Pf09a} is a coribbon WHA that is fusion and
weakly cofactorizable. If $H$ is a comodular WHA, then $\sym{M}^H$ is a
modular category~\cite{Pf09a}.

\section{Tannaka--Kre\v\i n reconstruction for a fusion categories}
\label{app_reconstruction}

\subsection{Fusion categories}

Let $\sym{C}$ be a multi-fusion category (see Appendix~\ref{sect_fusion}) over
some field $k$. By ${\{V_j\}}_{j\in I}$ we denote a (finite) set of
representatives of the isomorphism classes of simple objects of $\sym{C}$. We
use the small progenerator $\hat V=\bigoplus_{j\in I}V_j$. The \emph{long
canonical functor}
\begin{eqnarray}
\omega\colon\sym{C}\to\Vect_k,\quad X &\mapsto&\Hom(\hat V,\hat V\otimes X),\\
f &\mapsto& (\id_{\hat V}\otimes f)\circ-,\nn
\end{eqnarray}
is $k$-linear, faithful and exact~\cite{Ha99b}, takes values in
$\fdVect_k$, and has a separable Frobenius
structure~\cite{Sz05,Pf09a,Pf11}.

The algebra $R:=\End(\hat V)\cong\omega\one\cong k^{|I|}$ has a basis
${(\lambda_j)}_{j\in I}$ of orthogonal idempotents given by
$\lambda_j=\id_{V_j}\in R$, $j\in I$. It forms a Frobenius algebra
$(R,\circ,\id_R,\Delta_R,\epsilon_R)$ with comultiplication
$\Delta_R\colon R\to R\otimes R$ and counit $\epsilon_R\colon R\to k$
given by $\Delta(\lambda_j)=\lambda_j\otimes\lambda_j$ and
$\epsilon(\lambda_j)=1$ for all $j\in I$. The element $\Delta(\id_R)$
is a separability idempotent. Such a Frobenius algebra is called
\emph{index one} or \emph{Frobenius separable}~\cite{KaSz03}.

The category $\sym{C}$ is equipped with a family of non-degenerate
bilinear forms
\begin{equation}
\label{eq_gx}
g_X\colon\Hom(\hat V\otimes X,\hat V)\otimes\Hom(\hat V,\hat V\otimes X)\to k,\qquad
\theta\otimes v\mapsto\epsilon_R(\theta\circ v),
\end{equation}
which is associative, \ie\ compatible with composition,
$g_X((\theta\circ\omega f)\otimes v)=g_Y(\theta\otimes(\omega f\circ
v))$ for all $v\colon\hat V\to\hat V\otimes X$, $\theta\colon\hat
V\otimes Y\to\hat V$ and $f\colon X\to Y$. We use $g_X$ in order to
identify ${(\omega X)}^\ast\cong\Hom_k(\hat V\otimes X,\hat V)$.

By ${(e_m^{(X)})}_m$ and ${(e^m_{(X)})}_m$, we denote a pair of dual
bases of $\omega X=\Hom(\hat V,\hat V\otimes X)$ and $\Hom(\hat
V\otimes X,\hat V)$ with respect to $g_X$. In particular, we can
choose $e_j^{(\one)}=\rho_{\hat V}^{-1}\circ\lambda_j$ and
$e^j_{(\one)}=\lambda_j\circ\rho_{\hat V}$. Many computations are
particularly convenient if performed in these bases.

In addition to the dual basis lemma, \ie\ the triangle identities for
evaluation with the bilinear form $g_X$, the pair of dual bases satisfies
\begin{equation}
\label{eq_dualbasispair}
\sum_je^{(X)}_j\circ e^j_{(X)} = \id_{\hat V}\otimes\id_X
\end{equation}
for all simple $X\in|\sym{C}|$.

By a generalization of Tannaka--Kre\v\i n reconstruction from strong
monoidal functors to functors with separable Frobenius structure, we
obtain a finite-dimensional, split cosemisimple WHA
\begin{equation}
\label{eq_coendvect}
H = \coend(\sym{C},\omega) = \bigoplus_{j\in I}{(\omega V_j)}^\ast\otimes\omega V_j,
\end{equation}
such that $\sym{C}\simeq\sym{M}^H$ are equivalent as $k$-linear
additive monoidal categories. If $\sym{C}$ is fusion, \ie\ pure, then
in addition $H$ is copure, \ie\ $H_s\cap H_t\cong k$. The operations
of $H$ are given as follows~\cite{Pf09a},
\begin{eqnarray}
\mu({[\theta|v]}_X\otimes {[\zeta|w]}_Y)
&=& {[\zeta\circ(\theta\otimes\id_Y)\circ\alpha_{\hat V,X,Y}^{-1}|
\alpha_{\hat V,X,Y}\circ(v\otimes\id_Y)\circ w]}_{X\otimes Y},\\
\eta(1) &=& {[\rho_{\hat V}|\rho_{\hat V}^{-1}]}_\one,\\
\Delta({[\theta|v]}_X)
&=& \sum_j{[\theta|e_j^{(X)}]}_X\otimes {[e^j_{(X)}|v]}_X,\\
\epsilon ({[\theta|v]}_X) &=& \epsilon_R(\theta\circ v),\\
S({[e^j_{(X)}|e_\ell^{(X)}]}_X)
&=& {[\tilde e^\ell_{(X^\ast)}|\tilde e_j^{(X^\ast)}]}_{X^\ast},
\end{eqnarray}
where we write ${[\theta|v]}_X\in{(\omega X)}^\ast\otimes\omega X$
with $v\in\omega X$, $\theta\in{(\omega X)}^\ast$ and simple
$X\in|\sym{C}|$ for the homogeneous elements of $H$. The precise form
of the universal coend as a colimit also allows us to use the same
expression for arbitrary objects of $\sym{C}$, but subject to the
relations that $[{\zeta|(\omega f)(v)]}_Y ={[{(\omega
f)}^\ast(\zeta)|v]}_X$ for all $v\in\omega X$, $\zeta\in{(\omega
Y)}^\ast$ and for all morphisms $f\colon X\to Y$ of
$\sym{C}$. Recall that $(\omega f)(v)=(\id_{\hat V}\otimes f)\circ v$
and ${(\omega f)}^\ast(\zeta)=\zeta\circ(\id_{\hat V}\otimes
f)$. Furthermore, by ${(\tilde e_j^{(X^\ast)})}_j$ we denote the basis
of $\omega(X^\ast)$ defined by
\begin{equation}
\label{eq_antibasis}
\tilde e_j^{(X^\ast)} = \Psi_X(e^j_{(X)}),
\end{equation}
where
\begin{equation}
\label{eq_psi}
\Psi_X(\theta) = (\theta\otimes\id_{X^\ast})
\circ\alpha^{-1}_{\hat V,X,X^\ast}\circ(\id_{\hat V}\otimes\coev_X)
\circ\rho_{\hat V}^{-1}.
\end{equation}
By ${(\tilde e^j_{(X^\ast)})}_j$,we denote its dual basis with respect to the
bilinear form $g_{X^\ast}$, \cf~\eqref{eq_gx}.

The source and target counital maps are given by
\begin{eqnarray}
\epsilon_s({[\theta|v]}_X) &=&
{[\rho_{\hat V}|\rho^{-1}_{\hat V}\circ\theta\circ v]}_\one,\\
\label{eq_epsilont}
\epsilon_t({[\theta|v]}_X) &=&
\sum_\ell{[e^\ell_{(\one)}|\rho^{-1}_{\hat V}]}_\one\,
g_X(\theta\otimes(((\rho_{\hat V}\circ e^{(\one)}_\ell)\otimes\id_X)\circ v)),
\end{eqnarray}
for all simple $X\in|\sym{C}|$ and $\theta\in{(\omega X)}^\ast$, $v\in\omega X$.

Note that if $X\in|\sym{C}|$ is an arbitrary object, then $\omega X$
forms a right-$H$ comodule with the coaction
\begin{equation}
\beta_{\omega X}\colon\omega X\to\omega X\otimes H,\quad
v\mapsto\sum_je^{(X)}_j\otimes{[e^j_{(X)}|v]}_X.
\end{equation}
Its coefficient coalgebra is given by $C(X)={({(\omega
X)}^\ast\otimes\omega X)}/N_X\subseteq H$ where the subspace
$N_X\subseteq{(\omega X)}^\ast\otimes\omega X$ is generated by the
elements
\begin{equation}
{[\theta|(\omega f)(v)]}_X-{[{(\omega f)}^\ast(\theta)|v]}_X
\end{equation}
for all $v\in\omega X$, $\theta\in{(\omega X)}^\ast$ and
$f\in\End(X)$.

\subsection{Additional structure}

If $\sym{C}$ is pivotal with the monoidal natural isomorphism
$\tau_X\colon X\to{X^\ast}^\ast$, then $H=\coend(\sym{C},\omega)$ is
copivotal~\cite{Pf09b} with copivotal form and its convolution inverse
given by
\begin{eqnarray}
\label{eq_copivotal}
w({[\theta|v]}_X)
&=& g_X((D_{\hat V}^{-1}\circ\theta\circ(D_{\hat V}\otimes\id_X))\otimes v),\\
\bar w({[\theta|v]}_X)
&=& g_X(\theta\otimes((D_{\hat V}\otimes\id_X)\circ v\circ D_{\hat V}^{-1})),
\end{eqnarray}
for all $v\in\omega X$, $\theta\in{(\omega X)}^\ast$ and
$X\in|\sym{C}|$. Here
\begin{equation}
\label{eq_dtransformation}
D_{\hat V}=\sum_{j\in I}{(\dim V_j)}^{-1}\,\id_{V_j}\colon \hat V\to\hat V.
\end{equation}
More generally, there is a natural equivalence $D\colon 1_\sym{C}\Rightarrow
1_\sym{C}$ given by isomorphisms $D_X\colon X\to X$ for all $X\in|\sym{C}|$ as
follows. If
\begin{equation}
X\cong\bigoplus_{j\in I}m_jV_j
\end{equation}
with multiplicities $m_j\in\N_0$, then $D_X(v)=(\dim V_j)\,\id_{V_j}$ for all
homogeneous $v\in m_jV_j$.

If $\sym{C}$ is spherical, then $H$ is cospherical~\cite{Pf09b}. In
this case, the bilinear forms $g_X$ are related to the traces in
$\sym{C}$ as follows,
\begin{equation}
\label{eq_bilfspherical}
g_X(\theta\otimes v) = \epsilon_R(\theta\circ v) = \tr_{\hat V}(D_{\hat V}\circ\theta\circ v),
\end{equation}
for all $v\in\omega X$, $\theta\in{(\omega X)}^\ast$, $X\in|\sym{C}|$.

Furthermore, the basis dual to the $\tilde e^{(X^\ast)}_j$
of~\eqref{eq_antibasis} can be computed as
\begin{equation}
\tilde e^j_{(X^\ast)} = \Phi_X(e^{(X)}_j),
\end{equation}
where
\begin{equation}
\Phi_X(v) =
D^{-1}_{\hat V}\circ\rho_{\hat V}\circ(\id_{\hat V}\otimes\bar\ev_X)
\circ\alpha_{\hat V,X,X^\ast}\circ(v\otimes\id_{X^\ast})
\circ(D_{\hat V}\otimes\id_{X^\ast}),
\end{equation}
and~\eqref{eq_epsilont} simplifies to
\begin{equation}
\label{eq_epsilontnew}
\epsilon_t({[\theta|v]}_X) =
{[\Phi_X(v)\circ\Psi_X(\theta)\circ\rho_{\hat V}|\rho^{-1}_{\hat V}]}_\one
\end{equation}
with $\Psi_X$ as in~\eqref{eq_psi}.

If $\sym{C}$ is braided with braiding $\sigma_{X,Y}\colon X\otimes
Y\to Y\otimes X$, then $H$ is coquasi-triangular~\cite{Pf09a} with
universal $r$-form and its weak convolution inverse
\begin{eqnarray}
&&r({[\theta|v]}_X\otimes{[\zeta|w]}_Y)\nn\\
&=& g_{X\otimes Y}((\zeta\circ(\theta\otimes\id_Y)\circ\alpha^{-1}_{\hat V,X,Y})\otimes
((\id_{\hat V}\otimes\sigma_{Y,X})\circ\alpha_{\hat V,Y,X}\circ(w\otimes\id_X)\circ v),\\
&&\bar r({[\theta|v]}_X\otimes{[\zeta|w]}_Y)\nn\\
&=& g_{X\otimes Y}((\theta\circ(\zeta\otimes\id_X)\circ\alpha^{-1}_{\hat V,Y,Z}\circ(\id_{\hat V}\otimes\sigma^{-1}_{Y,X}))\otimes
(\alpha_{\hat V,X,Y}\circ(v\otimes\id_Y)\circ w)),
\end{eqnarray}
for all $v\in\omega X$, $w\in\omega Y$, $\theta\in{(\omega X)}^\ast$,
$\zeta\in{(\omega Y)}^\ast$ and $X,Y\in|\sym{C}|$.

If $\sym{C}$ is ribbon with twist $\nu_X\colon X\to X$, then $H$ is
coribbon~\cite{Pf09a} with universal ribbon form and its convolution
inverse given by
\begin{eqnarray}
\label{eq_nu}
\nu({[\theta|v]}_X)
&=& g_X(\theta\circ((\id_{\hat V}\otimes\nu_X)\circ v)),\\
\label{eq_nubar}
\bar\nu({[\theta|x]}_X)
&=& g_X(\theta\circ((\id_{\hat V}\otimes\nu^{-1}_X)\circ v)),
\end{eqnarray}
for all $v\in\omega X$, $\theta\in{(\omega X)}^\ast$, $X\in|\sym{C}|$.

Every ribbon category is pivotal, and in this case, the copivotal form
is given by $w(x)=v(x^\prime)\nu(x^\pprime)$ where $v\colon H\to k,
x\mapsto r(S(x^\prime)\otimes x^\pprime)$ denotes the second dual
Drinfel'd element. If $\sym{C}$ is modular, then $H$ is fusion,
coribbon and weakly cofactorizable~\cite{Pf09a}.

\newenvironment{hpabstract}{%
\renewcommand{\baselinestretch}{0.2}
\begin{footnotesize}%
}{\end{footnotesize}}%
\newcommand{\hpeprint}[2]{%
\href{http://www.arxiv.org/abs/#1}{\texttt{arxiv:#1#2}}}%
\newcommand{\hpspires}[1]{%
\href{http://www.slac.stanford.edu/spires/find/hep/www?#1}{SPIRES Link}}%
\newcommand{\hpmathsci}[1]{%
\href{http://www.ams.org/mathscinet-getitem?mr=#1}{\texttt{MR #1}}}%
\newcommand{\hpdoi}[1]{%
\href{http://dx.doi.org/#1}{\texttt{DOI #1}}}%
\newcommand{\hpjournal}[2]{%
\href{http://dx.doi.org/#2}{\textsl{#1\/}}}%

\end{document}